\documentclass[a4paper, 12pt]{amsart} 

\usepackage{amsmath,amssymb,enumitem,verbatim,stmaryrd,xcolor,microtype,graphicx,aliascnt,fancyvrb}

\usepackage[T1]{fontenc}
\usepackage[utf8]{inputenc}
\usepackage[english]{babel} 

\usepackage{mathptmx}                      

\usepackage[top=3.5cm,bottom=3.5cm,left=3.2cm,right=3.2cm]{geometry}
\usepackage[bookmarksdepth=2,linktoc=page,colorlinks,linkcolor={red!80!black},citecolor={red!80!black},urlcolor={blue!80!black},pdftitle={Strong congruence spaces and dimension in F1-geometry},pdfauthor={Manoel Jarra}]{hyperref}

\usepackage{stackrel}

\addto\extrasenglish{
    
}

\allowdisplaybreaks 

\usepackage{tikz}\usetikzlibrary{matrix,arrows,decorations.markings}
\usepackage{tikz-cd}

\usepackage[new]{old-arrows}

\usepackage{mathptmx}                                        
\usepackage{etoolbox}\makeatletter\patchcmd{\@startsection}{\@afterindenttrue}{\@afterindentfalse}{}{}\makeatother    
\patchcmd{\section}{\scshape}{\bfseries}{}{}\makeatletter\renewcommand{\@secnumfont}{\bfseries}\makeatother           
\usepackage[backgroundcolor=orange!30!white,linecolor=orange!80!white,textsize=footnotesize]{todonotes} \makeatletter \providecommand \@dotsep{5} \def\listtodoname{List of Todos} \def\listoftodos{\@starttoc{tdo}\listtodoname} \makeatother 

\addto\extrasenglish{}  

\DeclareRobustCommand{\gobblefour}[4]{}
\DeclareSymbolFont{sfoperators}{OT1}{bch}{m}{n}
\DeclareSymbolFontAlphabet{\mathsf}{sfoperators}
\makeatletter\def\operator@font{\mathgroup\symsfoperators}\makeatother 
\DeclareSymbolFont{cmletters}{OML}{cmm}{m}{it}
\DeclareSymbolFont{cmsymbols}{OMS}{cmsy}{m}{n}
\DeclareSymbolFont{cmlargesymbols}{OMX}{cmex}{m}{n}
\DeclareMathSymbol{\myjmath}{\mathord}{cmletters}{"7C}
\let\jmath\myjmath 
\DeclareMathSymbol{\myalpha}{\mathord}{cmletters}{"0B}
\let\alpha\myalpha 
\DeclareMathSymbol{\mybeta}{\mathord}{cmletters}{"0C} \let\beta\mybeta
\DeclareMathSymbol{\mygamma}{\mathord}{cmletters}{"0D} \let\gamma\mygamma
\DeclareMathSymbol{\mydelta}{\mathord}{cmletters}{"0E} \let\delta\mydelta
\DeclareMathSymbol{\myepsilon}{\mathord}{cmletters}{"0F}
\let\epsilon\myepsilon
\DeclareMathSymbol{\myzeta}{\mathord}{cmletters}{"10} \let\zeta\myzeta
\DeclareMathSymbol{\myeta}{\mathord}{cmletters}{"11} \let\eta\myeta
\DeclareMathSymbol{\mytheta}{\mathord}{cmletters}{"12} \let\theta\mytheta
\DeclareMathSymbol{\myiota}{\mathord}{cmletters}{"13} \let\iota\myiota
\DeclareMathSymbol{\mykappa}{\mathord}{cmletters}{"14} \let\kappa\mykappa
\DeclareMathSymbol{\mylambda}{\mathord}{cmletters}{"15} \let\lambda\mylambda
\DeclareMathSymbol{\mymu}{\mathord}{cmletters}{"16} \let\mu\mymu
\DeclareMathSymbol{\mynu}{\mathord}{cmletters}{"17} \let\nu\mynu
\DeclareMathSymbol{\myxi}{\mathord}{cmletters}{"18} \let\xi\myxi
\DeclareMathSymbol{\mypi}{\mathord}{cmletters}{"19} \let\pi\mypi
\DeclareMathSymbol{\myrho}{\mathord}{cmletters}{"1A} \let\rho\myrho
\DeclareMathSymbol{\mysigma}{\mathord}{cmletters}{"1B} \let\sigma\mysigma
\DeclareMathSymbol{\mytau}{\mathord}{cmletters}{"1C} \let\tau\mytau
\DeclareMathSymbol{\myupsilon}{\mathord}{cmletters}{"1D}
\let\upsilon\myupsilon
\DeclareMathSymbol{\myphi}{\mathord}{cmletters}{"1E} \let\phi\myphi
\DeclareMathSymbol{\mychi}{\mathord}{cmletters}{"1F} \let\chi\mychi
\DeclareMathSymbol{\mypsi}{\mathord}{cmletters}{"20} \let\psi\mypsi
\DeclareMathSymbol{\myomega}{\mathord}{cmletters}{"21} \let\omega\myomega
\DeclareMathSymbol{\myvarepsilon}{\mathord}{cmletters}{"22}\let\varepsilon\myvarepsilon
\DeclareMathSymbol{\myvartheta}{\mathord}{cmletters}{"23}
\let\vartheta\myvartheta
\DeclareMathSymbol{\myvarpi}{\mathord}{cmletters}{"24} \let\varpi\myvarpi
\DeclareMathSymbol{\myvarrho}{\mathord}{cmletters}{"25} \let\varrho\myvarrho
\DeclareMathSymbol{\myvarsigma}{\mathord}{cmletters}{"26}
\let\varsigma\myvarsigma
\DeclareMathSymbol{\myvarphi}{\mathord}{cmletters}{"27} \let\varphi\myvarphi

\newtheorem{thmA}{Theorem} 

\newaliascnt{conjA}{thmA}\aliascntresetthe{conjA}

\newaliascnt{propA}{thmA}\newtheorem{propA}[propA]{Proposition}\aliascntresetthe{propA}

\newaliascnt{corA}{thmA}\aliascntresetthe{corA}

\theoremstyle{plain}
\newtheorem{thm}{Theorem}[section] 
\newaliascnt{lemma}{thm}\newtheorem{lemma}[lemma]{Lemma}\aliascntresetthe{lemma}
\newaliascnt{cor}{thm}\newtheorem{cor}[cor]{Corollary}\aliascntresetthe{cor}
\newaliascnt{prop}{thm}\newtheorem{prop}[prop]{Proposition}\aliascntresetthe{prop}

\newtheorem*{notation}{Notation}
\newtheorem*{thm*}{Theorem}
\newtheorem*{lem*}{Lemma}
\newtheorem*{conj*}{Conjecture}
\newtheorem*{cor*}{Corollary}
\newtheorem*{problem*}{Problem}

\theoremstyle{definition}
\newaliascnt{df}{thm}\newtheorem{df}[df]{Definition}\aliascntresetthe{df}
\newaliascnt{rem}{thm}\newtheorem{rem}[rem]{Remark}\aliascntresetthe{rem}
\newaliascnt{ex}{thm}\newtheorem{ex}[ex]{Example}\aliascntresetthe{ex}

\newtheorem*{df*}{Definition}
\newtheorem*{ex*}{Example}
\newtheorem*{rem*}{Remark}

\usepackage{etoolbox}
\makeatletter
\patchcmd{\@startsection}{\@afterindenttrue}{\@afterindentfalse}{}{}             
\patchcmd{\part}{\bfseries}{\bfseries\LARGE}{}{}
\patchcmd{\section}{\scshape}{\bfseries}{}{}\renewcommand{\@secnumfont}{\bfseries} 
\patchcmd{\@settitle}{\uppercasenonmath\@title}{\large}{}{}
\patchcmd{\@setauthors}{\MakeUppercase}{}{}{}
\addto{\captionsenglish}{} 
\addto{\captionsenglish}{} 
\addto{\captionsenglish}{} 
\makeatother

\usepackage{fancyhdr}

\DeclareFontFamily{OT1}{pzc}{}                                
\DeclareFontShape{OT1}{pzc}{m}{it}{<-> s * [1.10] pzcmi7t}{}
\DeclareMathAlphabet{\mathpzc}{OT1}{pzc}{m}{it}
\DeclareSymbolFont{sfoperators}{OT1}{bch}{m}{n} \DeclareSymbolFontAlphabet{\mathsf}{sfoperators} \makeatletter\def\operator@font{\mathgroup\symsfoperators}\makeatother 
\DeclareSymbolFont{cmletters}{OML}{cmm}{m}{it}     

\DeclareMathOperator{\Spec}{Spec}
\DeclareMathOperator{\MSpec}{MSpec}
\DeclareMathOperator{\MSch}{MSch}
\DeclareMathOperator{\Cong}{Cong}
\DeclareMathOperator{\SCong}{SCong}

\DeclareMathOperator{\im}{im}
\DeclareMathOperator{\congker}{congker}

\newcommand\N{{\mathbb N}}
\newcommand\Z{{\mathbb Z}}
\newcommand\Q{{\mathbb Q}}
\newcommand\R{{\mathbb R}}
\newcommand\C{{\mathbb C}}
\newcommand\F{{\mathbb F}}
\newcommand\G{{\mathbb G}}
\newcommand\PL{{\mathbb P}}
\newcommand\T{{\mathbb T}}
\newcommand\A{{\mathbb A}}

\title{Strong congruence spaces and dimension in $\F_1$-geometry}

\author{Manoel Jarra}
\address{\rm Manoel Jarra, University of Groningen, the Netherlands, and IMPA, Rio de Janeiro, Brazil}
\email{{manoel.jarra@impa.br}}

\begin{document}

\begin{abstract}
We introduce strong congruence spaces, which are topological spaces that provide a useful concept of dimension for monoid schemes. We study their properties and show that, given a toric monoid scheme over an algebraically closed basis, its strong congruence space and the complex toric variety associated to its fan have the same dimension.
\end{abstract}

\maketitle

\begin{small} \tableofcontents \end{small}

\VerbatimFootnotes   
\thispagestyle{empty} 

\section*{Introduction}
\label{introduction}

Dimension theory for $\F_1$-geometry has not yet been developed in a satisfactory way. The main problem is that the underlying topological space of a monoid scheme can have too few points, so it fails to reflect the expected dimension. For example, if $\F$ is a pointed group, the underlying topological space of the $n$-torus $\T^n_{\F}$ is a singleton, so its Krull dimension (in the sense of \cite[Def. 5.10.1]{stacks-project}) is zero, while the dimension of the complex variety $\T^n_\C$ is $n$. We address this issue by attaching a new topological space to a monoid scheme $X$ over a basis $k$ whose Krull dimension behaves as expected if $k$ is an algebraically closed pointed group.

Throughout this text all monoids are assumed to be commutative. A \textit{congruence} on a pointed monoid $A$ is an equivalence relation $\mathfrak{c} \subseteq A\times A$ satisfying $(\lambda a, \lambda b) \in \mathfrak{c}$ whenever $(a, b) \in \mathfrak{c}$ and $\lambda \in A$, which implies that the projection $A \twoheadrightarrow A / \mathfrak{c}$ induces a natural structure of pointed monoid on the quotient. We say that a congruence $\mathfrak{c}$ is $\textit{prime}$ if $A / \mathfrak{c}$ is integral. We refer the reader to \autoref{preliminaries} for basic facts. Spaces of prime congruences were considered in \cite{Berkovich12} for monoids, in \cite{Deitmar13} for \textit{sesquiads} and more recently in \cite{Lorscheid-Ray23} for monoid schemes. The set of prime congruences of a  pointed monoid $A$ endowed with the topology generated by the open sets 
\[
U_{a, b} := \{ \mathfrak{c} \mid (a, b) \notin \mathfrak{c} \}
\]
is the \textit{congruence space} of $A$, denoted by $\Cong A$. For a monoid scheme $X = \bigcup \; \MSpec A_i$, the \textit{congruence space} of $X$ is constructed by gluing the affine pieces $\Cong A_i$, such that the result $X^\textup{cong}$ is independent of the choice of affine covering. 

We introduce the term \textit{domain} for a pointed monoid that can be embedded as a multiplicative submonoid into a field. \autoref{intrinsic domains} gives an intrinsic characterization of domains, without mention to fields. A prime congruence $\mathfrak{c}$ on a pointed monoid $A$ is $\textit{strong}$ if $A / \mathfrak{c}$ is a domain. If $A$ is a $k$-algebra, we say that $\mathfrak{c}$ is a \textit{$k$-congruence} if it is strong and the composition $k \rightarrow A \rightarrow A / \mathfrak{c}$ is injective. The set of $k$-congruences of $A$ with the topology induced from $\Cong A$ is denoted by $\SCong_k A$. Analogous to the case of prime congruences, we globalize this construction to the case of a $k$-scheme $X$ to define $\SCong_k X$, the \textit{strong $k$-congruence space} of $X$. 

An extension of pointed groups $\F \hookrightarrow \F'$ is \textit{algebraic} if for every $\alpha \in \F'$ there exists $n \geq 1$ such that $\alpha^n \in \F$. A pointed group $\F$ that is also a domain is called \textit{algebraically closed} if there is no non-trivial algebraic extension $\F \hookrightarrow \F'$ with $\F'$ being a pointed group and a domain. The following result, which is \autoref{thm-toric dimension}, shows that the strong congruence space of a toric monoid scheme over an algebraically closed basis reflects its expected dimension.

\begin{thmA}
Let $\F$ be an algebraically closed pointed group and $\Sigma$ a fan in a finite dimensional $\R$-vector space. Let $X(\Sigma)$ be the toric monoid scheme associated to $\Sigma$ and $Y$ the strong $\F$-congruence space of $X(\Sigma) \times_{\F_1}\MSpec \F$. Then $Y$ is a catenary space whose Krull dimension is the same as that of the complex toric variety $X(\Sigma)_\C$.
\end{thmA}

\subsection*{Relation to classical scheme theory} 
Let $R \mapsto R^\bullet$ be the forgetful functor from rings to pointed monoids. It has a left adjoint, given by $A \mapsto A_\Z := \Z[A] / \langle 1_\Z 0_A \rangle$ on objects. The inclusion 
\begin{align*}
A & \longrightarrow A_\Z^\bullet\\
a & \longmapsto 1_\Z  a
\end{align*}
induces a continuous map $p: \Spec (A_\Z) \rightarrow \MSpec A$ that sends a prime ideal $\mathfrak{p}$ of $\A_\Z$ to the prime ideal $\mathfrak{p} \cap A$ of $A$. The inclusion also induces a map $\gamma: \Spec (A_\Z) \rightarrow \SCong A$, such that $p$ factorizes through it, given by 
\[
\mathfrak{p} \mapsto \gamma(\mathfrak{p}) := \{(a, b) \in A \times A \mid \overline{a} = \overline{b} \in A_\Z / \mathfrak{p}\}.
\]

The construction above is globalized as follows: let $Y \mapsto Y^\bullet$ be the forgetful functor from Grothendieck schemes to monoidal spaces. Let $X$ be a monoid scheme and $\{\MSpec A_i\}_{i \in I}$ an affine open covering of $X$. Let $X_\Z = \textup{colim} \Spec(A_{i, \Z})$ be the base extension of $X$ to $\Z$, which is a Grothendieck scheme that does not depend on the choice of the covering $\{\MSpec A_i\}_{i \in I}$. There exists a morphism of monoidal spaces $X_\Z^\bullet \rightarrow X$ generalizing the map $p$ above (see \autoref{base extension to rings}).

\begin{thmA}
\label{intro factorization}
The underlying continuous map of the morphism of monoidal spaces $X_\Z^\bullet \rightarrow X$ factorizes through a surjective continuous map $\gamma: X_\Z \rightarrow \SCong X$ characterized by the property that if $U$ is an affine open of $X$, then $\gamma(U_\Z) \subseteq \SCong U$ and $\gamma|_{U_\Z}: U_\Z \rightarrow \SCong U$ is induced by the inclusion $\Gamma U \rightarrow (\Gamma U_\Z)^\bullet$ (see \autoref{factorization from schemes} and \autoref{XZ to SCong X is surjective}).
\end{thmA}

We have a description of the strong $k$-congruence space of a $k$-scheme in terms of a map coming from a Grothendieck scheme.

\begin{propA}
Let $k$ be a domain, $X$ a $k$-scheme and $F$ a field containing $k$ as a multiplicative submonoid. Let $\gamma_F$ be the composition $(X_\Z \times \Spec F)^\bullet \rightarrow X_\Z^\bullet \rightarrow \SCong X$. Then there exists a closed subscheme $Z$ of $X_\Z \times \Spec F$ whose image by $\gamma_F$ is $\SCong_k X$ (see \autoref{image = SCong_k}).
 \end{propA}

Even without a good sheaf on the space of (strong) congruences of a monoid scheme $X$, every point $x \in \SCong X$ has a meaningful ``residue field'' $\kappa(x)$ with the following property:

\begin{propA}
The factorization $\gamma$ presented in \autoref{intro factorization} induces an injective morphism of pointed groups $\gamma_x: \kappa(\gamma(x)) \rightarrow \kappa(x)^\bullet$ for every $x \in X_\Z$ (see \autoref{induced on the residue field by gamma}).
\end{propA}

In analogy with classical scheme theory, we say that a morphism of monoid schemes $\varphi: X \rightarrow Y$ is $\textit{locally of finite type}$ if every $x \in X$ has an affine open neighborhood $U$ and $\varphi(x)$ has an affine open neighborhood $V$ containing $\varphi(U)$ such that $\Gamma U$ is a finitely generated $\Gamma V$-algebra.

\begin{thmA}
Let $k$ be a domain and $X$ a $k$-scheme locally of finite type. Then $\SCong_k X$ is a sober space (see \autoref{strong congruence space is sober}).
\end{thmA}

\noindent \textbf{Follow-up work.} In \cite{Smirnov93} Smirnov explains that a conjectural Hurwitz inequality in $\F_1$-geometry implies the ABC conjecture. For this he finds ad hoc definitions for the ``compactification'' $\overline{\Spec \Z}$, for the projective line over algebraic extensions of $\F_1$ and for maps $\overline{\Spec \Z} \rightarrow \PL_{\F_1}^1$ induced by rational numbers. We show in \cite{Jarra23} how these spaces and maps fit into the framework of strong congruence spaces.


\subsection*{Acknowledgements}The author thanks Oliver Lorscheid for several conversations and for his help with preparing this text. The author also thanks Leo Herr, Alejandro Martínez Méndez and Eduardo Vital for useful conversations. The present work was carried out with the support of CNPq, National Council for Scientific and Technological Development - Brazil.

\section{Preliminaries}
\label{preliminaries}

In this section we recall the theory of monoid schemes. For more details, see \cite{Deitmar05}, \cite{Connes-Consani10}, \cite{Chuetal12} and \cite{Cortinasetal15}.

\subsection{Basic definitions}
\label{basic definitions}
A \textit{pointed monoid} is a set $A$ equipped with an associative binary operation
\begin{align*}
A \times A & \rightarrow A\\
(a, b) & \mapsto a\cdot b,
\end{align*}
called \textit{multiplication} or \textit{product}, such that $A$ has an element $0$, called \textit{absorbing element} or \textit{zero}, and an element $1$, called \textit{one}, satisfying $0\cdot a = a \cdot 0 = 0$ and $1 \cdot a = a \cdot 1 = a$ for all $a$ in $A$. We sometimes use $ab$ to denote $a \cdot b$. A pointed monoid is \textit{commutative} if $ab = ba$ for all $a, b$ in $A$. Throughout this text all monoids are commutative. A pointed monoid is \textit{integral} (or \textit{cancellative}) if $ab = ac$ implies $a = 0$ or $b = c$. A pointed monoid is \textit{without zero divisors} if $ab = 0$ implies $a = 0$ or $b = 0$. Note that if a pointed monoid is integral, then it is without zero divisors.

An \textit{ideal} of $A$ is a subset $I \subseteq A$
such that $0 \in I$ and $ax \in I$ for all $a \in A$ and $x \in I$. Every ideal $I$ induces an equivalence relation $\sim$ on $A$ generated by $\{(0, x) \in A\times A \mid x \in I\}$. The quotient set $A / I := A / \sim$ is a pointed monoid with operation given by $[a] \cdot [b]:=[ab]$, absorbing element $[0]$ and one $[1]$. The ideal $I$ is \textit{prime} if $A/I$ is without zero divisors. 

A \textit{morphism} of pointed monoids is a multiplicative map preserving zero and one. We use $\F_1$ to denote the initial object $\{0, 1\}$ of the category of pointed monoids. 

An element $u \in A$ is \textit{invertible} if there exists $y \in A$ such that $uy = 1$. We denote the set of invertible elements of $A$ by $A^\times$. The set $A \backslash A^\times$ is a prime ideal and contains every proper ideal of $A$. A morphism of pointed monoids $f:A \rightarrow B$ is \textit{local} if $f^{-1}(B^\times) = A^\times$. A \textit{pointed group} is a pointed monoid $A$ such that $A^\times = A \backslash\{0\}$. 

If $f: A \rightarrow B$ is a morphism and $I$ is an ideal of $B$, then the set 
\[
f^*(I):= \{a \in A \mid f(a) \in I\}
\]
is an ideal of $A$, called \textit{pullback of $I$ along $f$}. If $I$ is prime, then $f^*(I)$ is also prime.

A subset $S \subseteq A$ is \textit{multiplicative} if it is  multiplicatively closed and contains $1$. In this case, we define the \textit{localization} $S^{-1}A$ as follows: as a set, $S^{-1}A := (A \times S) / \sim$, where $\sim$ is the equivalence relation given by $(a, s) \sim (a', s')$ if there exists $t \in S$ such that $tsa' = ts'a$. We denote the class of $(a, s)$ in $S^{-1}A$ by $\frac{a}{b}$. The multiplication on $S^{-1}A$ is given by 
\[
\dfrac{a}{s} . \dfrac{b}{t} := \dfrac{ab}{st},
\]
with zero $\frac{0}{1}$ and one $\frac{1}{1}$. If $I$ is a prime ideal, then $S := A\backslash I$ is a multiplicative subset of $A$ and we denote the localization $S^{-1}A$ by $A_I$. For $h \in A$, we use $A[h^{-1}]$ to denote the monoid $\{h^n \mid n \in \N\}^{-1} A$.

If $A$ is without zero divisors, the \textit{group of fractions} $A$ is the localization $(A\backslash\{0\})^{-1} A$, denoted by $\textup{Frac}(A)$. It is a pointed group and the natural map $A \rightarrow \textup{Frac}(A)$ is injective.

Let $k$ be a pointed monoid. A \textit{$k$-algebra} is a pointed monoid $A$ with a morphism $k \rightarrow A$. We denote the image of an element $\lambda \in k$ via the map $k \rightarrow A$ also by $\lambda$. We say that $A$ is a \textit{finitely generated} $k$-algebra if there exists $a_1, \dotsc, a_n$ in $A$ such that for every $a \in A$ there exist $\lambda \in k$ and $m_1, \dotsc, m_n \in \N$ satisfying $a = \lambda a_1^{m_1}\cdots a_n^{m_n}$. A pointed monoid $A$ is \textit{finitely generated} if it is a finitely generated $\F_1$-algebra.

Let $M$ be a (not necessarily pointed) monoid with operation $+$. We define the $k$-algebra $k[M]$ as the set of formal monomials
\[
\big\{\lambda \chi^m \mid \lambda \in k\backslash\{0\} \text{ and } m \in M\big\} \cup \big\{0\big\}
\]
with operation $(\lambda\chi^m) \cdot (\mu \chi^n) := (\lambda \mu) \chi^{m+n}$ and $0 \cdot (\lambda\chi^m) = 0$. For $M = \N_{\geq 1}^n$ and variables $T_1, \dotsc, T_n$, we sometimes use $k[T_1, \dotsc, T_m]$ to denote $k[M]$, and $T_1^{m_1}\dotsc T_n^{m_n}$ for $\chi^{(m_1, \dotsc, m_n)}$.

\subsection{The geometry of monoids}
We use $\MSpec A$ to denote the set of prime ideals of the pointed monoid $A$. For $h \in A$, we denote the set $\{I \in \MSpec A \mid h \notin I\}$ by $U_h$. We endow $\MSpec A$ with the topology generated by the basis $\{U_a \mid a \in A\}$, and with the sheaf of pointed monoids $\mathcal{O}_{\MSpec A}$ characterized by $\Gamma(\mathcal{O}_{\MSpec A}, U_h) = A[h^{-1}]$ for $h \in A$, with restriction maps 
\[\def\arraystretch{1.5}
\begin{array}{ccc}
\Gamma(\mathcal{O}_{\MSpec A}, U_h) & \longrightarrow & \Gamma(\mathcal{O}_{\MSpec A}, U_{gh})\\
\dfrac{a}{h} & \longmapsto & \dfrac{ga}{gh}.
\end{array}
\]
A subset of the form $U_h$ is called \textit{principal open}.

A \textit{monoidal space} is a pair $(X, \mathcal{O}_X)$ where $X$ is a topological space and $\mathcal{O}_X$ is a sheaf of pointed monoids on $X$. A \textit{morphism of monoidal spaces} $(X, \mathcal{O}_X) \rightarrow (Y, \mathcal{O}_Y)$ is a pair $(\varphi, \varphi^\#)$ where $\varphi: X \rightarrow Y$ is a continuous map and $\varphi^\#: \mathcal{O}_Y \rightarrow \varphi_*\mathcal{O}_X$ is a morphism of sheaves of pointed monoids such that the induced map between stalks $\varphi^\#_x: \mathcal{O}_{Y, \varphi(x)} \rightarrow \mathcal{O}_{X, x}$ is local for every $x \in X$. 

An \textit{affine monoid scheme} is a monoidal space isomorphic to $(\MSpec A, \mathcal{O}_{\MSpec A})$ for some pointed monoid $A$. A \textit{monoid scheme} is a monoidal space that has an open covering by affine monoid schemes. As a full subcategory of monoidal spaces, monoid schemes form a category, denoted by $\MSch$.

Let $Y$ be a monoid scheme. An \textit{$Y$-scheme} is a monoid scheme $X$ with a morphism $X \rightarrow Y$. A \textit{morphism} of $Y$-schemes $f: X \rightarrow Z$ is a morphism of monoid schemes such that the diagram
\[
\begin{tikzcd}[column sep=1.5em]
X \arrow{rr}{f} \arrow{dr}{ }&& Z \arrow{dl} \\
 & Y
\end{tikzcd}
\]
commutes. We denote the category of $Y$-schemes by $\MSch / Y$. For the affine scheme $\MSpec k$, we write ``$k$-scheme'' and ``$\MSch / k$'' instead of ``$\MSpec k$-scheme'' and ``$\MSch / \MSpec k$'', respectively.

\subsection{Base extension to rings}
\label{base extension to rings}
Throughout this text all rings are commutative with unit. Let $R$ be a ring and $A$ a pointed monoid.

Define $A_R$ as the $R$-algebra $R[A]/\langle 1_R0_A \rangle$, where $R[A]$ is the monoid algebra of $A$ over $R$ and $\langle 1_R0_A \rangle$ is the ideal generated by $1_R0_A$\footnote{The subscripts at right indicate if the corresponding elements are in $R$ or $A$. In particular, $1_R$ is the unit of $R$ and $0_A$ is the absorbing element of $A$}. The assignment $A \mapsto A_R$ extends to a functor from the category of pointed monoids to $\textup{Alg}_R$, which is left-adjoint to the forgetful functor $B\mapsto B^\bullet$ from $R$-algebras to pointed monoids. For an affine monoid scheme $U = \MSpec A$, we use $U_R$ to denote the affine $R$-scheme $\Spec (A_R)$. We also define a functor $Y \mapsto Y^\bullet$ from schemes to monoidal spaces by simply forgetting the additive structure of $\mathcal{O}_Y$.

Let $X$ be a monoid scheme. Let $\mathcal{U}_X$ be the category of all affine open subschemes of $X$ with inclusions as morphisms. The \textit{base extension of $X$ to $R$} (or \textit{$R$-realization of $X$}) is the $R$-scheme
\[
X_R := \underset{U \in \mathcal{U}_X}{\textup{colim}} \; U_R.
\]
If $\{U_i\}_{i\in I}$ is an affine covering of a monoid scheme $X$, then $\{U_{i,R}\}_{i\in I}$ is an affine covering of $X_R$. If $U, V$ are affine open subsets of $X$, then $U_R \cap V_R = (U\cap V)_R$ in $X_R$ (\textit{cf.} \cite[Cor. 3.3]{Chuetal12} and \cite[Sec. 5]{Cortinasetal15}).

\begin{prop}
\label{proj from realization to monoid scheme}
The scheme $X_R$ comes with a morphism $p_R: X_R^\bullet \rightarrow X$ of monoidal spaces that is functorial in $X$.
\end{prop}

\begin{proof}
By \cite[Thm. 3.2]{Chuetal12}, one has a continuous map $\beta: X_\Z \rightarrow X$, functorial in $X$, characterized by $\beta^{-1}(U) = U_\Z$ for $U = \MSpec A$ affine open of $X$ and $\beta|_{U_\Z}: U_\Z \rightarrow U$ sending a prime ideal $\mathfrak{p} \in \Spec (A_\Z)$ to $\mathfrak{p} \cap A \in \MSpec A$. As $X_R = X_\Z \times_{\Spec \Z} \Spec R$, we define the underlying continuous map of $p_R$ as the composition 
\[
X_R^\bullet = X_R \longrightarrow X_\Z \xrightarrow{\;\;\beta\;\,} X.
\]
Note that $p_R^{-1}(U) = U_R$ for $U$\ affine open of $X$. One defines the morphism of sheaves $p_R^\#: \mathcal{O}_X \rightarrow p_{R,*}(\mathcal{O}_{X_R^\bullet})$ as the inclusion, characterized by
\[
p_R^\#(U): A = \Gamma(\mathcal{O}_X, U) \hookrightarrow A_R^\bullet = \Gamma(\mathcal{O}_{X_R^\bullet}, U_R)
\]
for $U = \MSpec A$ affine open of $X$.
\end{proof}

\section{Congruence spaces}


In this section we explain the construction of the congruence space associated to a monoid scheme as introduced in \cite[Section 2]{Lorscheid-Ray23}. 

\begin{df}
Let $A$ be a pointed monoid. A \textit{congruence} on $A$ is an equivalence relation $\mathfrak{c} \subseteq A\times A$ such that $(ab,ac) \in \mathfrak{c}$ whenever $(b,c) \in \mathfrak{c}$ and $a \in A$. The \textit{null ideal} of a congruence $\mathfrak{c}$ is the set
\[
I_\mathfrak{c} := \{ a \in A \mid (a,0) \in \mathfrak{c}\}.
\]

The \textit{congruence generated} by a subset $E = \{(a_i, b_i)\}_{i \in I} \subseteq A\times A$ is
\[
\langle E \rangle := \bigcap \{ \mathfrak{c} \mid  \mathfrak{c} \text{ is a congruence containing } E \},
\]
which is the smallest congruence containing $E$. The \textit{trivial congruence} is $\langle \emptyset \rangle$, which is the minimal congruence $\{(a, b) \in A \times A \mid a=b\}$.

The quotient set $A / \mathfrak{c}$ is a pointed monoid with product $[a] \cdot [b]:=[ab]$, absorbing element $[0]$ and one $[1]$. The congruence $\mathfrak{c}$ is called \textit{prime} if $A / \mathfrak{c}$ is integral. If $\mathfrak{c}$ is prime, then $I_\mathfrak{c}$ is a prime ideal. The \textit{congruence space of} $A$ is the set
\[
\Cong A := \{ \mathfrak{c} \subseteq A\times A \mid \mathfrak{c} \text{  is a prime congruence of } A\}
\]
 with the topology generated by the sets
\[
U_{a, b} := \{ \mathfrak{c} \in \Cong A \mid (a, b) \notin \mathfrak{c} \},
\]
for $a, b \in A$, \textit{i.e.}, the coarsest topology such that $U_{a, b}$ is open for all $a, b \in A$.

\end{df}

\begin{rem}
The set $\{U_{a,b} \mid a, b \in A\}$ is not a basis of the topology on $\Cong A$ in general (\textit{cf.} \cite[Example 2.10]{Lorscheid-Ray23}).
\end{rem}

\begin{notation}
\emph{For a congruence $\mathfrak{c}$, we sometimes use $a\sim_\mathfrak{c} b$ and $a \sim b$ instead of $(a, b) \in \mathfrak{c}$. We also use $\langle a_i \sim b_i \mid i \in I \rangle$ to denote the congruence $\langle \{(a_i, b_i)\}_{i \in I} \rangle$.}
\end{notation}

\begin{df}
Let $f: A \rightarrow B$ be a  morphism of pointed monoids. The \textit{congruence kernel} of $f$ is the set
\[
\textup{congker}(f) := \{ (a, b) \in A\times A \mid f(a) = f(b)\},
\]
which is indeed a congruence on $A$. Note that it is prime if $B$ is integral.

Let $\mathfrak{d}$ be a congruence on $B$. The \textit{pullback of $\mathfrak{d}$ along $f$} is the set
\[
f^*(\mathfrak{d}):= \{(x, y) \in A \times A \mid (f(x), f(y)) \in \mathfrak{d} \},
\]
which is a congruence on $A$. Note that $f^*(\mathfrak{d})$ is prime if $\mathfrak{d}$ is prime. The association $A\mapsto \Cong A$ for objects and $f \mapsto f^*$ for maps defines a contravariant functor from pointed monoids to topological spaces.

Let $\mathfrak{c}$ be a congruence on $A$. The \textit{push-forward of $\mathfrak{c}$ along $f$} is the congruence
\[
f_*(\mathfrak{c}) := \langle f(a) \sim f(b) \mid (a, b) \in \mathfrak{c} \rangle.
\]
\end{df}

\begin{prop} \cite[Prop. 2.8]{Lorscheid-Ray23}
\label{congruence-localization}
Let $A$ be a pointed monoid, $S$ a multiplicative subset of $A$ and $\iota_S: A \rightarrow S^{-1}A$ the localization map. Then the maps
 \[
\begin{tikzcd}[column sep=40pt]
\Cong S^{-1}A \ar[r,shift left=1.4,"\iota_S^\ast"] &
 \{\mathfrak{c} \in \Cong A \mid S \cap I_\mathfrak{c} = \emptyset\} \ar[l,shift
left=0.4,"\iota_{S,\ast}"]
\end{tikzcd}
\]
are mutually inverse bijections.
\end{prop}

\begin{notation}
\emph{If $\mathfrak{c}$ is a congruence on $A$, we use $S^{-1}\mathfrak{c}$ for the congruence $\iota_{S, *}(\mathfrak{c})$ on $S^{-1}A$.}
\end{notation}

Let $X$ be a monoid scheme and define $U^\textup{cong}:=\Cong \Gamma U$ for $U \in \mathcal{U}_X$. Given a morphism $U \rightarrow V$ in $\mathcal{U}_X$, one has an induced continuous map $U^\textup{cong} \rightarrow V^\textup{cong}$.

\begin{df}
\label{def congruence space}
The \textit{congruence space} of $X$ is the topological space
\[
X^\textup{cong} := \underset{U \in \mathcal{U}_X}{\textup{colim}} \, U^\textup{cong}.
\]
\end{df}

There is a functor from monoid schemes to topological spaces
\[
(-)^\textup{cong}: \MSch \rightarrow \textup{Top}
\]
extending $X\mapsto X^\textup{cong}$ such that for every morphism $\varphi: \MSpec B \rightarrow \MSpec A$ of affine monoid schemes one has $\varphi^\textup{cong} = \big(\varphi^\#(\MSpec A)\big)^*: \Cong B \rightarrow \Cong A$.

The space $X^\textup{cong}$ comes equipped with a continuous projection $\pi_X: X^\textup{cong} \rightarrow X$ given by $\pi_X(\tilde{x}) = I_\mathfrak{c}$ if $U \subseteq X$ is an affine open and $\mathfrak{c}$ is a prime congruence on $\Gamma U$ 
such that $\tilde{x} = \mathfrak{c} \in \Cong \Gamma U \subseteq X^\textup{cong}$.

\begin{prop}
Let $X$ be a monoid scheme and $\{U_i\}_{i \in I}$ an open covering of $X$. Then $\{U_i^\textup{cong}\}$ is an open covering of $X^\textup{cong}$.
\end{prop}

\begin{proof}
See \cite[Prop. 2.13]{Lorscheid-Ray23}.
\end{proof}

Let $X$ be a monoid scheme, $\tilde{x} \in X^\textup{cong}$ and $x:=\pi_X(\tilde{x}) \in X$. Let $\mathcal{U}_{X,x}$ be the full subcategory of $\mathcal{U}_X$ whose objects are the affine open neighborhoods of $x$. For each $U \in \mathcal{U}_{X,x}$, let $i_U: \Gamma U \rightarrow \mathcal{O}_{X, x}$ be the natural map into the colimit and $\mathfrak{c}_{\tilde{x},U} \in \Cong \Gamma U$ the prime congruence corresponding to $\tilde{x}$. We define the congruence
\[
\mathfrak{c}_{\tilde{x}} := \Bigg\langle \underset{U \in \mathcal{U}_x}{\bigcup} i_{U,*}(\mathfrak{c}_{\tilde{x},U}) \Bigg\rangle
\]
on $\mathcal{O}_{X,x}$, and the pointed group
\[
\kappa(\tilde{x}) := \mathcal{O}_{X,x} / \mathfrak{c}_{\tilde{x}}.
\]
By \cite[Lemma 2.17]{Lorscheid-Ray23}, $\mathfrak{c}_{\tilde{x}}$ is prime and for any $U \in \mathcal{U}_{X,x}$ one has $\mathfrak{c}_{\tilde{x}} = i_{U,*}(\mathfrak{c}_{\tilde{x},U})$ and $\kappa(\tilde{x}) \simeq \textup{Frac}(\Gamma U  / \mathfrak{c}_{\tilde{x},U})$.

A morphism of monoid schemes $\varphi: X \rightarrow Y$ induces an  injective morphism of pointed groups 
\[
\kappa\big(\varphi^\textup{cong}(\tilde{x})\big) \hookrightarrow \kappa(\tilde{x})
\]
for every  $\tilde{x} \in X^\textup{cong}$.

\section{The strong congruence space}

\begin{df}
\label{df-domain}
A \textit{domain} is a pointed monoid $A$ that is isomorphic to a pointed submonoid of $K^\bullet$ for some field $K$. A \textit{strong} prime congruence on $A$ is a congruence $\mathfrak{c}$ such that $A/ \mathfrak{c}$ is a domain. We denote the set of strong prime congruences on A by $\SCong A$. 

Let $k$ be a pointed monoid and $A$ a $k$-algebra. A \textit{$k$-congruence} on $A$ is a strong prime congruence $\mathfrak{c} \in \SCong A$ such that the natural map $k \rightarrow A/ \mathfrak{c}$ is injective. The set of $k$-congruences of $A$ is denoted by $\SCong_k A$. Note that $\SCong A = \SCong_{\F_1} A$ and $\SCong_k A$ is empty if $k$ is not a domain or if the map $k \rightarrow A$ is not injective.

A \textit{$k$-prime} of $A$ is a prime ideal $P \in \MSpec A$ such that the natural map $k\rightarrow A/P$ is injective. The set of $k$-primes of $A$ is denoted by $\MSpec_k A$. Note that if $\mathfrak{c} \in \SCong_k A$, then $I_\mathfrak{c} \in \MSpec_k A$.
\end{df}

\begin{rem}
If $A$ is a domain, then $A$ is integral and for every $\alpha$ in $A$ and $n \geq 1$ one has $\#\{x\in A|\, x^n = \alpha \}\leq n$. \autoref{intrinsic domains} shows that these two properties characterize domains.
\end{rem}

\begin{rem}
If $A$ is a domain, then any pointed submonoid of $\textup{Frac}(A)$ is a domain.
\end{rem}

\begin{rem}
Throughout this text the term ``domain'' only refers to pointed monoids as presented in \autoref{df-domain}, and not to nonzero commutative rings without zero divisors.
\end{rem}

\begin{lemma}
\label{k-localization}
Let $A$ be a $k$-algebra and $S$ a multiplicative subset of $A$. Then the map $\mathfrak{c} \mapsto S^{-1}\mathfrak{c}$ gives a bijection 
\[
\{\mathfrak{c} \in \SCong_k A \mid I_\mathfrak{c} \cap S = \emptyset\} \xrightarrow{\sim} \SCong_k S^{-1}A.
\]
\end{lemma}

\begin{proof}
Let $\mathfrak{c}$ be a $k$-congruence on $A$ with $I_\mathfrak{c} \cap S = \emptyset$. Then there exists a field $K$ and an inclusion $A/ \mathfrak{c} \hookrightarrow K^\bullet$. As $K^\bullet$ is a pointed group, the composition $A \rightarrow A / \mathfrak{c} \rightarrow K^\bullet$ induces a map $S^{-1}A \rightarrow K^\bullet$ whose kernel is $S^{-1}\mathfrak{c}$. Thus $S^{-1}\mathfrak{c} \in \SCong S^{-1}A$. As the diagram
\[
\begin{tikzcd}
k \arrow{r}{} \arrow{rd}{} & A \arrow{r}{} \arrow[swap]{d}{ } & A/ \mathfrak{c} \arrow{d}{ } \arrow{r}{} & K^\bullet\\%
 & S^{-1}A \arrow{r}{ }  & S^{-1}A/ S^{-1}\mathfrak{c} \arrow{ru}{} 
\end{tikzcd}
\]
commutes and the map $k \rightarrow K^\bullet$ given by the upper row is injective, we conclude that $S^{-1}\mathfrak{c} \in \SCong_k A$.

By \autoref{congruence-localization}, the map $\mathfrak{c} \mapsto S^{-1}\mathfrak{c}$ is injective and given $\mathfrak{d}$ in $\SCong_k S^{-1}A$, there exists $\mathfrak{r} \in \Cong A$ such that $\mathfrak{d} = S^{-1}\mathfrak{r}$. It only remains to show that $\mathfrak{r}$ is a $k$-congruence. The composition
\[
k \rightarrow A \rightarrow A/\mathfrak{r} \rightarrow S^{-1}\big(A/\mathfrak{r}\big) \simeq S^{-1}A / S^{-1}\mathfrak{r}
\]
is injective, thus $k \rightarrow A/ \mathfrak{r}$ is injective. As $A/\mathfrak{r}$ is integral and $S^{-1}A / S^{-1}\mathfrak{r}$ is a domain, $A/\mathfrak{r}$ is a domain.
\end{proof}

\begin{lemma}
Let $A$ be a $k$-algebra, $\mathfrak{d}$ a congruence on $A$ and $\pi: A \twoheadrightarrow A / \mathfrak{d}$ the natural projection. Then the map $\mathfrak{c} \mapsto \pi_*(\mathfrak{c})$ gives a bijection
\[
\{\mathfrak{c} \in \SCong_k A \mid \mathfrak{d} \subseteq \mathfrak{c}\} \xrightarrow{\sim} \SCong_k (A/\mathfrak{d}).
\]
\end{lemma}

\begin{proof}
By \cite[Prop. 4.13]{Lorscheid-Ray23}, the map $\mathfrak{c} \mapsto \pi_*(\mathfrak{c})$ gives a bijection
\[
\{\mathfrak{c} \in \Cong A \mid \mathfrak{d} \subseteq \mathfrak{c}\} \xrightarrow{\sim} \Cong (A/\mathfrak{d}).
\]
As $A / \mathfrak{c} \simeq (A / \mathfrak{d}) \big/ \pi_*(\mathfrak{c})$ for any congruence $\mathfrak{c}$ with $\mathfrak{d} \subseteq \mathfrak{c}$, the result follows.
\end{proof}

\begin{lemma}
\label{strong pullback}
Let $k$ be a domain, $f : A \rightarrow B$ a morphism of $k$-algebras and $\mathfrak{q}$ a $k$-congruence on $B$. Then $f^*(\mathfrak{q})$ is a $k$-congruence of $A$. This shows that the map $f^*$ restricts to a continuous map $f^*: \SCong_k B \rightarrow \SCong_k A$.
\end{lemma}

\begin{proof}
Note that $f$ induces an injection $A/f^*(\mathfrak{q}) \hookrightarrow B/ \mathfrak{q}$. As $B/ \mathfrak{q}$ is a domain, one has that $A/f^*(\mathfrak{q})$ is also a domain. As the diagram
\[
\begin{tikzcd}
 & A \arrow{r}{} \arrow{d}{ } & B \arrow{d}{ } \\%
k \arrow{r}{} \arrow{ru}{} & A/ f^*(\mathfrak{q}) \arrow[hook]{r}{ }  & B/ \mathfrak{q}
\end{tikzcd}
\]
commutates and the map $k \rightarrow B/ \mathfrak{q}$ given by the bottom row is injective, one has that $k \rightarrow A/ f^*(\mathfrak{q})$ is injective. Thus $f^*(\mathfrak{q}) \in \SCong_k A$.
\end{proof}

Let $k$ be a pointed monoid and $X$ a $k$-scheme. For $U, V \subseteq X$ affine open with $U \subseteq V$, let $\iota_{UV}: U \hookrightarrow V$ be the inclusion. Let $\mathcal{U}_{X, k}^\textup{scong}$ be the category whose objects are the topological spaces $U_k^\textup{scong}:= \SCong_k \Gamma U$ for affine open subschemes $U$ of $X$, with
\vspace{1mm}
\[
\textup{Hom}_{\mathcal{U}_{X, k}^{\textup{cong}}}(U_k^\textup{scong},V_k^\textup{scong}) = \smash{\left\{\begin{array}{cl}
      \{(\Gamma\iota_{UV})^*\} &\text{if } U\subseteq  V\\
    {}\emptyset &\text{if } U\not\subseteq V.
    \end{array}\right.}
\]

\begin{df}
 The \textit{strong $k$-congruence space} of $X$ is the topological space
\[
\SCong_k X := \textup{colim} \, \mathcal{U}_{X, k}^\textup{scong}.
\]
\end{df}

Note that $\SCong_k X$ comes with a natural inclusion to $X^\textup{cong}$. We denote the restriction of the projection $\pi_X: X^\textup{cong} \rightarrow X$ to $\SCong_k X$ again by $\pi_X$.

\begin{lemma}
\label{principal intersecion}
Let $U$ and $V$ be affine open subschemes of a monoid scheme $X$ and $x \in U\cap V$. Then there exists an affine open neighborhood $W \subseteq U\cap V$ of $x$ that is a principal open subset of both $U$ and $V$.
\end{lemma}

\begin{proof}
Let $W \subseteq U \cap V$ be an affine open neighborhood of $x$. Let $A:= \Gamma U$. By \cite[Lemma 2.4]{Cortinasetal15}, there exists $\mathfrak{p} \in \MSpec A$ such that $W = \MSpec A_\mathfrak{p}$. By \cite[Lemma 1.3]{Cortinasetal15}, $W$ is a principal open of $U$. 

Analogously, $W$ is a principal open of $V$.
\end{proof}

\begin{lemma}
\label{strong covering}
Let $X$ be a $k$-scheme and $\{U_i\}_{i \in I}$ a collection of open subschemes that covers $X$. Then $\{\SCong_k U_i\}_{i \in I}$ is an open covering of $\SCong_k X$. 
\end{lemma}

\begin{proof}
As $\{U_i^\textup{cong}\}_{i \in I}$ is an open covering of $X_k^\textup{cong}$, it is enough to show that if $U$ is an open subscheme of $X$, then $\SCong_k U = U^\textup{cong} \cap \SCong_k X$. 

Let $x \in U^\textup{cong} \cap \SCong_k X$. Then there exists $V = \MSpec M \subseteq X$ affine open such that $x \in V_k^\textup{scong}$. Let $W = \MSpec N \subseteq U$ affine open such that $x\in W^\textup{cong}$. By \autoref{principal intersecion}, there exists $Y = \MSpec A \subseteq W \cap V$ affine open containing $\pi_X(x)$ that is a principal open of both $W$ and $V$.
Thus $x \in Y^\textup{cong} \subseteq (W \cap V)^\textup{cong}$. Let $s \in M$ such that $M[s^{-1}] \simeq A$. One has $x \in \SCong_k M \cap \Cong M[s^{-1}]$. By \autoref{k-localization}, this implies 
\[
x \in \SCong_k M[s^{-1}] = Y_k^\textup{scong} \subseteq \SCong_k U.
\]
\end{proof}

\begin{prop}
There is an extension of the association $X \mapsto \SCong_k X$ to a subfunctor of $(-)^\textup{cong}: \MSch / k \rightarrow \textup{Top}$.
\end{prop}

\begin{proof}\
One knows that $\SCong_k X$ is a subspace of $X^\textup{cong}$ for any $X \in \MSch / k$.

Next we define the functor $\SCong_k(-)$ on morphisms. Let $\varphi: X \rightarrow Y$ be a morphism in $\MSch / k$. Let $\{U_\alpha\}_{\alpha \in \Lambda}$ be an affine open covering of $X$ and $\{V_\alpha\}_{\alpha \in \Lambda}$ an affine open covering of $Y$ such that $\varphi(U_\alpha) \subseteq V_\alpha$ for any $\alpha \in \Lambda$. Let $\varphi_\alpha: U_\alpha \rightarrow V_\alpha$ be the restriction of $\varphi$. Note that $\varphi_\alpha$ is induced by a morphism of pointed monoids $\gamma_\alpha: \Gamma V_\alpha \rightarrow \Gamma U_\alpha$. By \autoref{strong pullback}, $\gamma_\alpha^\textup{cong}(U_{\alpha, k}^\textup{scong}) \subseteq V_{\alpha, k}^\textup{scong}$. By \autoref{strong covering}, $\{U_{\alpha, k}^\textup{scong}\}_{\alpha \in \Lambda}$ is an open covering of $\SCong_k X$ and $\{V_{\alpha, k}^\textup{scong}\}_{\alpha \in \Lambda}$ is an open covering of $\SCong_k Y$. Thus 
\[
\varphi^\textup{cong}(\SCong_k X) = \underset{\alpha \in \Lambda}{\bigcup} \gamma_\alpha^\textup{cong}(U_{\alpha, k}^\textup{scong}) \subseteq \SCong_k Y.
\]
Define $\SCong_k \varphi: \SCong_k X \rightarrow \SCong_k Y$ as the restriction of $\varphi^\textup{cong}$.

The fact that $\SCong_k(-)$ is a subfunctor of $(-)^\textup{cong}$ follows from the construction above.
\end{proof}

\section{Relation to classical scheme theory}
\label{s}

\subsection{The map \texorpdfstring{$\gamma: X_R \rightarrow \SCong X$}{gamma: XR ---> SCong X}} We construct a continuous map $\gamma$ such that the underlying continuous map of $p_R: X_R^\bullet \rightarrow X$ factorizes as 
\[
\begin{tikzcd}
X_R \arrow{r}{\gamma} \arrow[swap]{rd}{p_R} & \SCong X \arrow{d}{\pi_X}\\
 & X.
\end{tikzcd}
\]

Let $A$ be a pointed monoid and $R$ a ring. A morphism of pointed monoids $\varphi: A\rightarrow R^\bullet$ induces a continuous map $\varphi^*: \Spec R \rightarrow \SCong A$ that sends a prime ideal $\mathfrak{p}$ to the strong prime congruence 
\[
\varphi^*(\mathfrak{p}):= \congker(\pi_\mathfrak{p}^\bullet \circ \varphi) = \big\{(a,b)\in A\times A \;\big|\; \overline{\varphi(a)} = \overline{\varphi(b)} \text{ in } R/\mathfrak{p}\big\},
\]
where $\pi_\mathfrak{p}: R \rightarrow R/\mathfrak{p}$ is the natural projection.

\begin{thm}
\label{factorization from schemes}
Let $X$ be a monoid scheme and $R$ a ring. The underlying continuous map of the morphism of monoidal spaces $p_R: X_R^\bullet \rightarrow X$ factorizes through a continuous map $\gamma: X_R \rightarrow \SCong X$ characterized by the property that if $U$ is an affine open of $X$, then $\gamma(U_R) \subseteq \SCong U$ and $\gamma|_{U_R}: U_R \rightarrow \SCong U$ is induced by the inclusion $\Gamma U \rightarrow (\Gamma U_R)^\bullet$.
\end{thm}

\begin{proof}
As $X_R = X_\Z \times_\Z \Spec R$, it is enough to show the result for $R = \Z$. The result follows for $X$ affine by \autoref{proj from realization to monoid scheme} and the discussion above. We just need to verify that if $U$ and $V$ are affine open subsets of $X$, then the diagram
\[
\begin{tikzcd}
U_\Z \cap V_\Z \arrow{r}{} \arrow[swap]{d}{} & U_\Z \arrow{d}{} \\%
V_\Z \arrow{r}{}& \SCong X
\end{tikzcd}
\]
commutes. By \cite[Cor. 3.3]{Chuetal12}, $U_\Z \cap V_\Z = (U \cap V)_\Z$ and, by \autoref{principal intersecion}, the intersection $U \cap V$ has a covering by principal open subschemes of both $U$ and $V$. As the diagram
\[
\begin{tikzcd}
U_\Z \arrow[swap]{d}{} & \arrow{l} W_\Z \arrow{r}{}\arrow[swap]{d}{} & V_\Z \arrow[swap]{d}{} \\
\SCong U  \arrow[swap]{dr}{} & \arrow{l}{} \SCong W \arrow{r}{} \arrow{d} & \arrow[swap]{ld} \SCong V \\
 & \SCong X & 
\end{tikzcd}
\]
commutes for every $W$ that is a principal open subscheme of both $U$ and $V$, the result follows.
\end{proof}

\begin{prop}
\label{induced on the residue field by gamma}
The factorization of $p_R$ through $\gamma$, as in \autoref{factorization from schemes}, induces an injective morphism of pointed groups $\gamma_x: \kappa(\gamma(x)) \rightarrow \kappa(x)^\bullet$ for every $x \in X_R$.
\end{prop}

\begin{proof}
Let $x \in X_R$ and $U = \MSpec A \subseteq X$ an affine open subscheme such that $x \in U_R$. Let $\mathfrak{p}_x$ be the prime ideal of $A_\Z$ corresponding to $x$. As $\gamma(x)$ is defined as the point corresponding to the congruence kernel $\mathfrak{c}$ of the map 
\[A \rightarrow A_\Z^\bullet \twoheadrightarrow (A_\Z / \mathfrak{p}_x)^\bullet,
\]
one has an injective morphism of pointed groups 
\[
\kappa(\gamma(x)) = \textup{Frac}(A / \mathfrak{c}) \rightarrow \big(\textup{Frac}(A_\Z / \mathfrak{p}_x)\big)^\bullet = \kappa(x)^\bullet.
\]
\end{proof}

\subsection{Characterization of \texorpdfstring{$\SCong_k X$}{SCongk X}} We give a characterization of $\SCong_k X$ in terms of a continuous map coming from the base extension of $X$ to a field that contains $k$ as a multiplicative submonoid.

\begin{lemma}
\label{bound on roots of 1 for the group of fractions}
Let $A$ be an integral pointed monoid such that $\#\{x \in A \mid x^n = \alpha\} \leq n$ for all $\alpha \in A$ and $n \geq 1$. Then $\#\{y \in \textup{Frac}(A) \mid y^n = 1\} \leq n$ for all $n \geq 1$.
\end{lemma}

\begin{proof}
Let $n \geq 1$ and $a_1, \dotsc, a_{n+1}, b_1, \dotsc, b_{n+1} \in A \backslash \{0\}$ such that $\big(\frac{a_i}{b_i}\big)^n = 1$ for all $i = 1, \dotsc, n+1$. Define $\alpha:= (b_1 \cdots b_{n+1})^n$. As $a_i^n = b_i^n$, one has
\[
\alpha = \Bigg(\underset{i \neq j}{\prod}b_i \cdot a_j \Bigg)^n
\]
for all $1 \leq j \leq n+1$. Thus there exists $\ell \neq m$ such that 
\[
\underset{i \neq \ell}{\prod}b_i \cdot a_\ell = \underset{i \neq m}{\prod}b_i \cdot a_m.
\]
As $A$ is integral, one has $b_m a_\ell = b_\ell a_m$, which implies $\frac{a_\ell}{b_\ell} = \frac{a_m}{b_m}$.
\end{proof}

\begin{lemma}
\label{prop-surjection}
Let $K$ be a field, $A$ a pointed submonoid of $K^\bullet$ and $S$ a (not necessarily pointed) monoid. Let $\mathfrak{c}$ be a prime congruence on $A[S]$\footnote{See \autoref{basic definitions} for the definition of $A[S]$.} such that:
\begin{enumerate}
    \item The natural map $A\rightarrow A[S]/\mathfrak{c}$ is injective;
    \item $\#\{x\in A[S]/\mathfrak{c} \; |\; \, x^n = \alpha \}\leq n$ for every $\alpha$ in $A[S]/\mathfrak{c}$ and $n \geq 1$.
\end{enumerate}
Then there exists $\mathfrak{p} \in \Spec K[S]$ such that $\mathfrak{c}=i^*(\mathfrak{p})$, where $i: A[S] \rightarrow (K[S])^\bullet$ is the natural inclusion.
\end{lemma}

\begin{proof}
As $\mathfrak{c}$ is prime, $M := (A[S]/\mathfrak{c}) \backslash \{0\}$ is a monoid. Let $G$ be the Grothendieck group of $M$. Note that $A\backslash \{0\} \rightarrow G$ is injective. Let $I$ be the ideal  of $K[G]$ generated by 
\[
\big\{1_K \lambda_G - \lambda_K 1_G \in K[G] \; \big| \; \lambda \in A \big\}, 
\]
where, for $a$ in $A$, the subscript in $a_K$ (resp. $a_G$) indicates $a$ as an element of $K$ (resp. $G$). Define the following maps: 

\begin{enumerate}
    \item The natural inclusion $i: A[S]\hookrightarrow K[S]^\bullet$;
    \item $f: A[S] \rightarrow K[G]^\bullet$, where $f(x) = 1 \overline{x}$ for $x\neq 0$, and $f(0) = 0$;
    \item The morphism of $K$-algebras $g: K[S] \rightarrow K[G] / I$ induced by $s\mapsto 1\overline{s} + I$ for $s\in S$, where $\overline{s}$ is the class of $s$ in $G$;
    \item The natural projection $\pi: K[G] \twoheadrightarrow K[G]/I$.
\end{enumerate}

We have a commutative diagram of pointed monoids
\[
\begin{tikzcd}
A[S] \arrow{r}{i} \arrow[swap]{d}{f} & (K[S])^\bullet \arrow{d}{g^\bullet} \\%
(K[G])^\bullet \arrow{r}{\pi^\bullet}& (K[G]/I)^\bullet.
\end{tikzcd}
\]

We show that there exists a field $E$ and a morphism of rings $\psi: K[G] \rightarrow E$ satisfying $I\subseteq \ker(\psi)$ such that the induced map $z\mapsto \psi(1z)$ from $G$ to $E$ is injective. Thus $\mathfrak{p} := g^{-1}\big(\pi(\ker(\psi)\big)$ is a prime ideal of $K[S]$ satisfying $i^*(\mathfrak{p}) = \mathfrak{c}$.

Let $\Gamma$ be the set $2^{K[G]}$ and $E$ the algebraic closure of the field $K(T_\gamma |\, \gamma \in \Gamma)$. Let $\mathcal{H}$ be the set of $K$-algebra morphisms $\varphi: K[H_\varphi] \rightarrow E$ satisfying:

\begin{enumerate}[label=(\roman*)]
    \item $H_\varphi$ is a subgroup of $G$ containing $A\backslash \{0\}$;
    \item The induced map $z\mapsto \varphi(1z)$ from $H_\varphi$ to $E$ is injective;
    \item $\{1_K \lambda_G - \lambda_K 1_G \in K[H_\varphi] \mid \lambda \in A\} \subseteq \ker(\varphi)$.
\end{enumerate}

Let $\leq$ be the partial order on $\mathcal{H}$ given by 
$\varphi \leq \psi$ if $H_\varphi \subseteq H_\psi$ and $\psi\vert_{K[H_\varphi]} = \varphi$. Note that $\mathcal{H}$ is non-empty, because there exists a morphism of $K$-algebras $\tau: K[\langle A\backslash \{0\} \rangle] \rightarrow  E$ such that $\tau(1_Ka_G) = a_K1_G$, where $\langle A\backslash \{0\} \rangle$ is the subgroup of $G$ generated by $A\backslash \{0\}$. Let $\mathcal{C}$ be a chain in $\mathcal{H}$. Define $H_\beta := \underset{\varphi \in \mathcal{C}}{\bigcup} H_\varphi$ and note that $H_\beta$ is a group. For $\varphi, \psi \in \mathcal{C}$, if $y \in K[H_\varphi] \cap K[H_\psi]$ one has $\varphi(y) = \psi(y)$. Thus the map 
\[
\begin{array}{cccl}
H_\beta & \longrightarrow & E & \\
x & \longmapsto & \varphi(x) &\text{if } x\in K[H_\varphi],
\end{array}
\]
induces an element $\beta \in \mathcal{H}$. Note that $\beta$ is an upper bound for $\mathcal{C}$. By Zorn's lemma, $\mathcal{H}$ has a maximal element $\Phi$. Next we show that $H_\Phi = G$. 

Suppose that there is an $h\in G\backslash H_\Phi$. We have three cases to analyze:

\medskip\noindent
\textbf{Case 1:} $\{\ell > 0 \mid h^\ell \in H_\Phi\} = \emptyset$, \textit{i.e.}, $h$ is not a torsion element and $\langle h \rangle \cap H_\Phi = \{1_G\}$. In this case, $H_\Omega := \langle H_\Phi, h \rangle$ is isomorphic to $H_\Phi \times \Z$. As $|\Gamma| > \big|K[G]\big|$, there exists $\alpha \in \Gamma$ such that $T_\alpha \notin \im \Phi$. Let $\Omega: K[H_\Omega] \rightarrow E$ be the morphism of $K$-algebras such that $\Omega\vert_{K[H_\Phi]} = \Phi$ and $\Omega(1h) = T_\alpha$. Therefore, $\Omega \in \mathcal{H}$ and $\Phi < \Omega$, which contradicts the maximality of $\Phi$.

\medskip\noindent
\textbf{Case 2:} $\{\ell > 0 \mid  h^\ell \in H_\Phi\} \neq \emptyset$ and $h$ is a torsion element. Then there exist
\[
w := \min \{\ell > 0 \mid h^\ell \in H_\Phi\} \quad \text{and} \quad s := \textup{ord}(h^w)
\]
such that $sw = \textup{ord}(h)$. Note that $\Phi(1h^w)$ is a root of $1$ in $E$ and let $\zeta \in E$ be a primitive $w^{\textup{th}}$-root of $\Phi(1h^w)$. Note that $\textup{ord}(h^w) = \textup{ord}(\Phi(1h^w))$. Let $H_\Omega$ be the subgroup $\langle H_\Phi, h \rangle$. There exists a morphism of groups $\omega: H_\Omega \rightarrow E^*$ such that $\omega(x)  = \Phi(1x)$ for $x\in H_\Phi$, and $\omega(h) = \zeta$. Next we show that $\omega$ is injective.
    
If $\omega$ is not injective, then there exists $y\in H_\Phi$ and $m\in \{1, \dotsc, w-1\}$ such that $\omega(yh^{-m}) = 1$, which implies $\Phi(y) = \omega(y) = \omega(h^m) = \zeta^m$. Then 
\[
\Phi(1 y^{sw}) = \zeta^{msw} = \omega((h^{sw})^m) = \omega(1) = \Phi(1),
\]
which implies $y^{sw} = 1$. As $\textup{ord}(h) = sw$, by \autoref{bound on roots of 1 for the group of fractions}, one has $y \in \langle h \rangle$. As the intersection $\langle h \rangle \cap H_\Phi$ is equal to $\langle h^w \rangle$, there exists $t>0$ such that $y = h^{wt}$. This implies $\zeta^{wt-m} = \omega(yh^{-m}) = 1$ and, consequently, that $sw=\textup{ord}(\zeta)$ divides $wt-m$, which is a contradiction because $w$ does not divide $m$. 

Let $\Omega: K[H_\Omega] \rightarrow E$ be the morphism of $K$-algebras induced by $\omega$. Note that $\Omega \in \mathcal{H}$ and $\Phi < \Omega$, which contradicts the maximality of $\Phi$.

\medskip\noindent
\textbf{Case 3:} $\{\ell > 0 \mid  h^\ell \in H_\Phi\} \neq \emptyset$ and $h$ is not a torsion element. Let 
\[
w := \min \{\ell > 0 \mid  h^\ell \in H_\Phi\}
\]
and fix $\xi\in E$ such that $\xi^w = \Phi(1h^w)$. Suppose that $\langle \xi \rangle \cap \Phi(1H_\Phi) \neq \langle \xi^w \rangle$. Define $m := \min\{\ell > 0 \mid \xi^\ell \in \Phi(1 H_\Phi)\}$. Then $w > m$ and exists $y\in H_\Phi$ such that $\Phi(1y) = \xi^m$. Write $w=qm+r$ with $q,r$ integers satisfying $0 \leq r <m$. Note that $h^wy^{-q} \in H_\Phi$ and $\Phi(1h^wy^{-q}) = \xi^r$. As $r<m$ and $\Phi\vert_{1H_\Phi}$ is injective, $r=0$ and $h^w = y^q$. Let $\tilde{h} := h^my^{-1}$. Note that $\tilde{h} \notin H_\Phi$ and $\tilde{h}^q = 1$. Analogously to \textbf{Case 2}, we get a contradiction. Thus $\langle \xi \rangle \cap \Phi(1H_\Phi) = \langle \xi^w \rangle$. 

Let $H_\Omega$ be the subgroup $\langle H_\Phi, h \rangle$. There exists an injective morphism of groups $\omega: H_\Omega \rightarrow E^*$ such that $\omega(x)  = \Phi(1x)$ for $x\in H_\Phi$, and $\omega(h) = \xi$. Let $\Omega: K[H_\Omega] \rightarrow E$ be the morphism of $K$-algebras induced by $\omega$. Note that $\Omega \in \mathcal{H}$ and $\Phi < \Omega$, which contradicts the maximality of $\Phi$.
\end{proof}

\begin{prop}
\label{equivalence for domain}
Let $A$ be an integral pointed monoid such that $\#\{x\in A|\, x^n = \alpha \}\leq n$ for every $\alpha$ in $A$ and $n \geq 1$. Then there exists a field $K$ of characteristic zero and an injective morphism of pointed monoids $A\hookrightarrow K^\bullet$. 
\end{prop}

\begin{proof}
Note that $A$ is isomorphic to $\F_1[A]/\mathfrak{c}$, where $\mathfrak{c}:= \langle 0_{\F_1} 1_A \sim 1_{\F_1} 0_A \rangle$. As $\F_1$ is a pointed submonoid of $\C^\bullet$, by \autoref{prop-surjection}, there exists $\mathfrak{p} \in \Spec \, \C[A]$ such that $\mathfrak{c} = i^*(\mathfrak{p})$, where $i: \F_1[A]\rightarrow \C[A]^\bullet$ is the inclusion. 

Let $K$ be the field $\textup{Frac}(\C[A]/\mathfrak{p})$. Then the composition
\[
A\simeq \F_1[A]/\mathfrak{c} \rightarrow (\C[A]/\mathfrak{p})^\bullet \rightarrow K^\bullet
\]
is an injective morphism $A\hookrightarrow K^\bullet$ of pointed monoids.
\end{proof}

The next result, which follows from \autoref{equivalence for domain}, gives an \textit{intrisic} (i.e., without mention to fields) description of domains.

\begin{cor}
\label{intrinsic domains}
Let $A$ be a pointed monoid. Then $A$ is a domain if, and only if, $A$ is integral and satisfies $\#\{x\in A|\, x^n = \alpha \}\leq n$ for every $\alpha$ in $A$ and $n \geq 1$.
\end{cor}

\begin{rem}
\label{equivalence for domain - char zero}
By \autoref{equivalence for domain}, every domain is isomorphic to a pointed submonoid of $K^\bullet$ for some field $K$ \textit{of characteristic zero}.
\end{rem}

\begin{prop}
\label{surjectiveness from Spec}
Let $B$ be a pointed monoid, $F$ a field and $k$ a domain that is a pointed submonoid of both $B$ and $F^\bullet$. Let $R:=  F[B] / I$, where $I$ is the ideal generated by $\{ \lambda_k  1_B - 1_k \lambda_B \mid \lambda\in k \}$. Then the map $\Spec R\rightarrow \SCong B$, induced by $b \mapsto \overline{1b}$, which sends $\mathfrak{p} \in \Spec R$ to the congruence kernel of the composition $B \rightarrow R^\bullet \rightarrow (R/\mathfrak{p})^\bullet$, has image $\SCong_k B$. 
\end{prop}

\begin{proof}
Let $\varphi: B\rightarrow R^\bullet$ given by $\varphi(b):= \overline{1b}$ and $\mathfrak{p} \in \Spec R$. As $F$ is a field, the composition $F \rightarrow R \rightarrow R /\mathfrak{p}$ is injective. As 
\[
\begin{tikzcd}
k \arrow[r, hook] \arrow[d, hook] & B \arrow{d}{\varphi} \\%
F^\bullet \arrow[r, hook]& R^\bullet.
\end{tikzcd}
\]
commutes, one has $\varphi^*(\mathfrak{p}) \in \SCong_k B$. Therefore $\im \varphi^* \subseteq \SCong_k B$. Next we prove that $\SCong_k B \subseteq \im \varphi^*$.

Let $\mathfrak{d}$ be a $k$-congruence on $B$. Define the congruence
\[
\mathfrak{a} := \langle \{\lambda_k  1_B \sim 1_k \lambda_B \mid \lambda\in k\}  \cup  \{1 a \sim 1 b \mid (a,b)\in \mathfrak{d}\rangle
\]
 on $k[B]$. The map $\overline{b} \mapsto \overline{1b}$ defines an isomorphism $B/ \mathfrak{d} \xrightarrow{\sim} k[B]/\mathfrak{a}$. As $B/ \mathfrak{d}$ is a domain, one has $\mathfrak{a} \in \SCong k[B]$. By \autoref{prop-surjection}, there exists $\mathfrak{q} \in \Spec F[B]$ such that $\mathfrak{a} = i^*(\mathfrak{q})$, where $i: k[B] \rightarrow (F[B])^\bullet$ is the inclusion.

Let $\pi: F[B] \rightarrow R$ be the natural projection. As $\lambda_k 1_B \sim_\mathfrak{a} 1_k \lambda_B$, one has
\[
\lambda_k 1_B - 1_k \lambda_B \in \mathfrak{q} \quad \text{for all} \quad \lambda \in k.
\]
This implies $I\subseteq \mathfrak{q}$ and, consequently, $\pi(\mathfrak{q}) \in \Spec R$. As the diagram
\[
\begin{tikzcd}
B \arrow{r}{\varphi} \arrow[swap]{d}{ } & R^\bullet \arrow{d}{ } \\
B/\mathfrak{d} \arrow{r}{ } \arrow["\rotatebox{90}{\(\sim\)}"]{d}{ } & \big(R / \pi(\mathfrak{q})\big)^\bullet \\
k[B]/\mathfrak{a} \arrow[hook]{r}{ } & (F[B]/\mathfrak{q})^\bullet \arrow[swap, "\rotatebox{90}{\(\sim\)}"]{u}{ }
\end{tikzcd}
\]
commutes, $B/\mathfrak{d} \rightarrow \big(R / \pi(\mathfrak{q})\big)^\bullet$ is injective and $\varphi^*\big(\pi(\mathfrak{q})\big) = \mathfrak{d}$.
\end{proof}

\begin{lemma}
\label{tensor product}
Let $A$ be a pointed monoid, $B$ an $A$-algebra and $R$ a ring such that $R^\bullet$ is an $A$-algebra. Let $I$ be the ideal of $R[B]$ generated by $\{\lambda_R 1_B - 1_R \lambda_B \mid \lambda\in A \}$. Then the rings $R[B] / I$ and $B_\Z \otimes_{A_\Z} R$ are isomorphic.
\end{lemma}

\begin{proof}
The map $b \mapsto (1b) \otimes 1$ from $B$ to $B_\Z \otimes_{A_\Z} R$ induces a morphism of $R$-algebras $f: R[B] \rightarrow B_\Z \otimes_{A_\Z} R$. As $I \subseteq \ker f$, one has a morphism of $R$-algebras  \[
  \begin{array}{cccc}
  \overline{f}: & R[B] / I & \longrightarrow & B_\Z \otimes_{A_\Z} R \\
           & \overline{1b} & \longmapsto & (1b) \otimes 1.
  \end{array}
 \]

As the map $(\sum n_i b_i , r) \mapsto \overline{\sum (n_i r) b_i}$ from $B_\Z \times R$ to $R[B] / I$ is $A_\Z$-bilinear, it induces a morphism of $A_\Z$-algebras
 \[
  \begin{array}{cccc}
  g: & B_\Z \otimes_{A_\Z} R & \longrightarrow & R[B] / I \\
           & (\sum n_i b_i)  \otimes r & \longmapsto & \overline{\sum (n_i r) b_i}.
  \end{array}
 \]

 As $\overline{f}$ and $g$ are mutually inverse, the result follows.
\end{proof}

\begin{thm}
\label{image = SCong_k}
Let $k$ be a domain, $X$ a $k$-scheme and $F$ a field such that $k$ is a pointed submonoid of $F^\bullet$. Then $X_\Z \times_{k_\Z} \Spec F$ is a closed subscheme  of $X_F$ whose image by $\gamma$ is $\SCong_k X$.
\end{thm}

\begin{proof}
For each $U = \MSpec B$ affine open of $X$, define
\[
I_U := \langle \lambda_k 1_B - 1_k \lambda_B \mid \lambda\in k \rangle \triangleleft B_F. 
\]
Note that $B_F / I_U \simeq F[B] / I_U'$, where $I_U' := \langle \lambda_k 1_B - 1_k \lambda_B \mid \lambda\in k \rangle \triangleleft F[B]$. Let $\mathcal{I}_U$ be the sheaf of ideals on $U_F$ induced by $I_U$. If $U$ and $V$ are affine open subsets of $X$ and $W \subseteq U\cap V$ is a principal open of both $U$ and $V$, then $\mathcal{I}_U|_W = \mathcal{I}_W = \mathcal{I}_V|_W$. Thus there exists a sheaf of ideals $\mathcal{I}$ on $X_F$ such that $\mathcal{I}|_{U_F} = \mathcal{I}_U$ for every affine open $U \subseteq X$.

Let $Z$ be the closed subscheme of $X_F$ corresponding to $\mathcal{I}$. If $U\subseteq X$ is an affine open, one has $Z \times_{X_F} U_F = \Spec (B_F / I_U)$ and, by \autoref{surjectiveness from Spec}, 
\[
\gamma(Z \times_{X_F} U_F) = \SCong_k U \subseteq \SCong X. 
\]
By \autoref{tensor product}, $Z = X_\Z \times_{k_\Z} \Spec F$ and, by \autoref{strong covering}, $\gamma(Z) = \SCong_k X$.
\end{proof}

\begin{cor}
\label{XZ to SCong X is surjective}
The map $\gamma: X_\Z \rightarrow \SCong X$ is surjective.
\end{cor}

\begin{proof}
As $\F_1 \subseteq \Q$ and $\SCong X = \SCong_{\F_1} X$, by \autoref{image = SCong_k}, the map $X_\Q \rightarrow \SCong X$ is surjective. As the diagram
\[
\begin{tikzcd}
X_\Q \arrow[r] \arrow[dr, two heads] &
X_\Z \arrow[d]
\\
& \SCong X
\end{tikzcd}
\]
commutes, the result follows.
\end{proof}

\subsection{The topology of \texorpdfstring{$\SCong_k X$}{SCongk X}} We prove that if $X$ is a $k$-scheme locally of finite type, then $\SCong_k X$ is a sober space, \textit{i.e.}, every irreducible closed subset of $\SCong_k X$ has a unique generic point.

\begin{thm}
\label{generic points}
Let $k$ be a domain and $B$ a finitely generated $k$-algebra. Then $\SCong_k B$ is sober.
\end{thm}

\begin{proof}
If $k \rightarrow B$ is not injective, then $\SCong_k B = \emptyset$ has no irreducible closed subsets.

Suppose that $k \rightarrow B$ is injective. Let $F$ be a field with an inclusion $k \hookrightarrow F^\bullet$ and define 
\[
R:= F[B]/ \langle 1_F \lambda_B - \lambda_F 1_B \mid \lambda\in k \rangle. 
\]
Let $\varphi: B \rightarrow R^\bullet$ be given by $b \mapsto \overline{1b}$. By \autoref{surjectiveness from Spec}, $\varphi^*: \Spec R \rightarrow \SCong B$ has image $\SCong_k B$.

Let $Z\subseteq\SCong_k B$ be an irreducible closed subset. Then $(\varphi^*)^{-1}(Z)$ is a closed subset of $\Spec R$. As $R$ is a Noetherian ring, there exists $W_1, \dotsc, W_n \subseteq \Spec R$ irreducible closed subsets such that $(\varphi^*)^{-1}(Z) = W_1 \cup \cdots \cup W_n$. For each $i$, let $\eta_i \in \Spec R$ be the generic point of $W_i$. In $\SCong_k B$, one has
\[
Z = \overline{\varphi^*\big((\varphi^*)^{-1}(Z)\big)} = \underset{i=1}{\overset{n}{\bigcup}} \; \overline{\varphi^*(W_i)}.
\]
As $Z$ is irreducible, there exists $\ell$ such that $Z = \overline{\varphi^*(W_\ell)}$. Note that $(\varphi^*)^{-1} \big( \overline{\{\varphi^*(\eta_\ell)\}} \big)$ is a closed subset of $\Spec R$ that contains $\eta_\ell$. Thus $W_\ell \subseteq (\varphi^*)^{-1} \big( \overline{\{\varphi^*(\eta_\ell)\}} \big)$, which implies
\[
\varphi^*(W_\ell) \subseteq \varphi^* \bigg( (\varphi^*)^{-1} \big(\overline{\{\varphi^*(\eta_\ell)\} }\big) \bigg) = \overline{\{\varphi^*(\eta_\ell)\} }
\]
and, consequently, $Z = \overline{\varphi^*(W_\ell)} \subseteq \overline{\{\varphi^*(\eta_\ell)\} }$. Therefore $\varphi^*(\eta_\ell)$ is a generic point of $Z$.

As $\Cong B$ is a $T_0$ space and $\SCong_k B$ is a subspace of $\Cong B$, one has uniqueness.
\end{proof}

\begin{df}
A morphism of monoid schemes $\varphi: X \rightarrow Y$ is \textit{locally of finite type} if $Y$ has an affine open covering $\{\MSpec B_i\}_{i \in I}$ such that, for each $i \in I$, the subscheme $\varphi^{-1}(\MSpec B_i)$ has an affine open covering $\{\MSpec A_{ij}\}_{j \in J_i}$ with $A_{ij}$ finitely generated $B_i$-algebra for all $j \in J_i$.

An $Y$-scheme $X$ is \textit{locally of finite type} if the map $X \rightarrow Y$ is locally of finite type.
\end{df}

\begin{thm}
\label{strong congruence space is sober}
Let $k$ be a domain and $X$ a $k$-scheme locally of finite type. Then $\SCong_k X$ is sober.
\end{thm}

\begin{proof}
Let $\varphi: X \rightarrow \MSpec k$ be the morphism that makes $X$ a $k$-scheme and $Z$ an irreducible closed subset of $\SCong_k X$. Let $U = \MSpec B$ be an affine open of $X$ such that $Z \cap \SCong_k U \neq \emptyset$. Then there exists $\{f_i\}_{i \in I} \subseteq B$ such that $\{\MSpec \; B[f_i^{-1}]\}_{i \in I}$ covers $U$ and, for each $i$, there exists $\{g_{ij}\}_{j \in J_i} \subseteq k$, such that $\{\MSpec \; k[g_{ij}^{-1}]\}_{i \in I}$ covers $\varphi^{-1}(\MSpec B[f_i^{-1}])$ and with $B[f_i]$ a finitely generated $k[g_{ij}]$-algebra for all $j \in J_i$. Note that $B[f_i]$ is a finitely generated $k$-algebra for all $i \in I$.

Let $\ell \in I$ such that $Z_\ell := Z \cap \SCong_k B[f_\ell^{-1}]$ is not empty. As $Z_\ell$ is an irreducible closed subset of $\SCong_k B[f_\ell^{-1}]$, by \autoref{generic points}, $Z_\ell$ has a unique generic point $\eta$. As $Z_\ell$ is a non-empty open subset of $Z$, one has that $\eta$ is the generic point of $Z$. 
\end{proof}

\section{Strong congruence space of toric monoid schemes}

\subsection{Algebraically closed pointed groups} We introduce algebraically closed pointed groups and gather some results about them (cf. \cite[pp. 519-520]{Deitmar08}).

\begin{df}
An \textit{extension} of pointed groups is an injective morphism $f: \G_1 \hookrightarrow \G_2$. In this case we identify $\G_1$ with its image via $f$. The extension is \textit{trivial} if it is an isomorphism. An element $\lambda \in \G_2$ is \textit{algebraic} over $\G_1$ if there exists $n \geq 1$ such that $\lambda^n \in \G_1$. We say that $\G_2$ is \textit{algebraic over} $\G_1$ if every $x\in \G_2$ is algebraic over $\G_1$. 
\end{df}

\begin{df}
\label{algebraically closed}
A pointed group $\G$ is \textit{algebraically closed} if for every $\alpha$ in $\G^\times$ and $n \geq 1$ one has $\#\{x\in \G|\, x^n = \alpha \}$ equal to $n$.
\end{df}

\begin{rem}
By \autoref{intrinsic domains}, every algebraically closed pointed group is a domain.
\end{rem}

\begin{prop}
\label{algebraic closure}
Let $\G$ be a pointed group that is a domain. Then there exists an algebraically closed pointed group  $\overline{\G}$ and an algebraic extension $\iota: \G \hookrightarrow \overline{\G}$ such that every extension of pointed groups $f: \G \rightarrow \G'$ with $\G'$ algebraically closed factorizes as 
\[
\begin{tikzcd}
\G \arrow[d, "\iota", swap]\arrow{r}{f} & \G'\\
\overline{\G} \arrow[ur, "j", swap]
\end{tikzcd}
\]
with $j: \overline{\G} \rightarrow \G'$ injective. We call $\overline{\G}$ an \textit{algebraic closure of $\G$}.
\end{prop}

\begin{proof}
By \autoref{equivalence for domain}, there exists a field $K$ of characteristic zero and an injective morphism $\varphi: \G \hookrightarrow K^\bullet$ (see \autoref{equivalence for domain - char zero}). Let $L$ be an algebraic closure of $K$ (in the usual field-theoretic sense) and define
\[
\overline{\G} := \{\lambda \in L \mid \exists n \geq 1 \text{ such that } \lambda^n \in \G\}.
\]
Note that $\overline{\G}$ is an algebraically closed pointed group and $\im \varphi \subseteq \overline{\G}$. The map 
\[
\begin{array}{cccc}
\iota: & \G & \longhookrightarrow & \overline{\G} \\
           & \alpha & \longmapsto & \varphi(\alpha)
\end{array}
\]
makes $\overline{\G}$ algebraic over $\G$.

 Let $f: \G \rightarrow \G'$ be an extension of pointed groups with $\G'$ algebraically closed. We say that a morphism $g: \G_g \rightarrow \G'$ is an \textit{extension} of $f$ if it is injective, $\G_g$ is a pointed subgroup of $\overline{\G}$ containing $\G$ and $f$ factorizes through $g$. Let $\mathcal{E}$ be the set of all extensions of $f$, with the partial order given by $g \leq h$ if $\G_g \subseteq \G_h$ and $h|_{\G_g} = g$. By Zorn's lemma, $\mathcal{E}$ has a maximal element $z$. Next we show that $\G_z = \overline{\G}$.

Suppose that there exists $\lambda \in \overline{\G}\backslash \G_z$. Let $m := \textup{min}\{n\geq 1 \mid \lambda^n \in \G_z\}$ and $x := \lambda^m$. Let $\widetilde{\G}$ be the pointed submonoid of $\overline{\G}$ generated by $\G_z$ and $\lambda$. As $\G'$ is algebraically closed and $\widetilde{\G} \simeq \G_z[T] / \langle T^m \sim x \rangle$, there exists $y \in \G' \backslash \textup{im}(z)$ with $y^m = x$ and an extension $r: \widetilde{\G} \rightarrow \G'$ of $z$ sending $\lambda$ to $y$. Then $r \in \mathcal{E}$ and $r>z$, which is a contradiction. Therefore $\G_z = \overline{\G}$.
\end{proof}

\begin{cor}
Let $\G$ be a pointed group that is also a domain. If $\G_1$ and $\G_2$ are two algebraic closures of $\G$, then $\G_1 \simeq \G_2$.
\end{cor}

\begin{proof}
By \autoref{algebraic closure}, there exists an injective morphism of $\G$-algebras $\varphi: \G_1 \rightarrow \G_2$. Let $\lambda \in \G_2$. As $\lambda$ is algebraic over $\G$, there exists $m := \textup{min}\{n\geq 1 \mid \lambda^n \in \G\}$. Let $L_i := \{x \in \G_i \mid x^m = \lambda^m\}$, for $i = 1, 2$. As $\varphi(L_1) \subseteq L_2$ and $\#L_1 = m = \#L_2$, one has $\varphi(L_1) = L_2$. As $\lambda \in L_2$, we conclude that $\lambda \in \textup{im}\, \varphi$.
\end{proof}

\begin{ex}
By the proof of \autoref{algebraic closure}, an algebraic closure of $\F_1$ is 
\[
\F_{1^\infty} = \{z \in \C \mid \exists n\geq 1 \text{ such that } z^n = 1 \} \cup \{0 \}.
\]
\end{ex}

\begin{prop}
Let $\G$ be a pointed group that is also a domain. Then $\
\G$ is algebraically closed if, and only if, there is no non-trivial algebraic extension $\G \hookrightarrow \G'$ with $\G'$ a pointed group that is also a domain.
\end{prop}

\begin{proof}
Assume that $\G$ is algebraically closed. Let $\G \hookrightarrow \G'$ be an algebraic extension of pointed groups with $\G'$ a domain. If $\lambda \in (\G')^\times$, then there exists $n \geq 1$ such that $\alpha := \lambda^n \in \G$. Note that $I := \{x \in \G \mid x^n = \alpha\}$ has $n$ elements. As $I$ is contained in $J := \{y \in \G' \mid y^n = \alpha\}$, which has at most $n$ elements, one has $I = J$. As $\lambda \in J$, we conclude that $\lambda \in \G$. Thus $\G = \G'$, \textit{i.e.}, the extension $\G \hookrightarrow \G'$ is trivial.

The opposite direction follows from \autoref{algebraic closure}.
\end{proof}

\subsection{Main theorem}

We prove that the strong congruence space of a toric monoid scheme over an algebraically closed pointed group reflects its expected dimension.

We use the following notations from \cite{Brasselet01} and \cite{Eur15} for toric geometry. For a torsion-free abelian group $L$ of finite rank, we use $L_\R$ to denote the vector space $L\otimes_\Z\R$. Throughout this subsection, fix $M$ and $N$ dual torsion-free abelian groups of finite rank $n$ with perfect pairing $\langle \;, \; \rangle: M\times N \rightarrow \Z$, which extends to an $\R$-linear pairing $\langle \;, \; \rangle: M_\R \times N_\R \rightarrow \R$. A $\textit{rational polyhedral cone}$ in $N_\R$ is a set of the form 
\[
\sigma = \text{Cone}(z_1, \dotsc, z_\ell) := \{\lambda_1 z_1 + \cdots + \lambda_\ell z_\ell \in N_\R \mid \lambda_i \in \R_{\geq 0}, \text{ for all } i\}
\]
for some $z_1, \dotsc, z_\ell \in N$, where we identify $N$ with its image in $N_\R$ via $x\mapsto x\otimes 1$. The \textit{dimension} of $\sigma$ is the dimension of the linear subspace generated by $\sigma$. The \textit{dual cone} of $\sigma$ is the set 
\[
\sigma^\vee := \{x \in M_\R \mid \langle x, y \rangle \geq 0, \text{ for all } y \in \sigma \}.
\]
It is a fact that $\sigma^\vee$ is a rational polyhedral cone in $M_\R$ and $(\sigma^\vee)^\vee = \sigma$. A \textit{face} of $\sigma$ is a set of the form
\[
\beta = \sigma \cap H_m, \quad \text{where} \quad H_m:= \{u \in N_\R \mid \langle m, u \rangle = 0\},
\]
for some $m \in \sigma^\vee$. Any face of $\sigma$ is a rational polyhedral cone. We say that $\sigma$ is \textit{strongly convex} if $\dim \sigma^\vee = \dim M_\R = n$. Equivalently, $\sigma$ is strongly convex if it does not contain any non-zero linear subspace. 

A \textit{fan} is a finite collection $\Sigma$ of rational polyhedral cones in $N_\R$ satisfying:
\begin{enumerate}
    \item Every $\sigma \in \Sigma$ is strongly convex;
    \item Every face of a cone in $\Sigma$ is also in $\Sigma$;
    \item If $\sigma_1, \sigma_2 \in \Sigma$, then $\beta:= \sigma_1 \cap \sigma_2$ is a face of both $\sigma_1$ and $\sigma_2$.
\end{enumerate}
We define the \textit{dimension} of $\Sigma$ as $\dim \Sigma := \dim N_\R = n$.

Given a rational polyhedral cone $\sigma \subseteq N_\R$, we define the affine semigroups
\[
S_\sigma := \sigma^\vee \cap M \quad \text{and} \quad S_\sigma^\vee := \sigma \cap N,
\]
which are finitely generated monoids. Note that $S_{(\sigma^\vee)} = S_\sigma^\vee$.  

\begin{rem}
The notations $\dim \Sigma$ and $S_\sigma^\vee$ are introduced in the present text.
\end{rem}

Given a fan $\Sigma$, let $\mathcal{U}$ be the category whose objects are the elements of $\Sigma$ and the morphisms are the inclusions of faces. Given two cones $\tau \subseteq \sigma$ in $\Sigma$, the induced map 
\[
\MSpec \F_1[S_\tau] \rightarrow \MSpec \F_1[S_\sigma]
\]
is an open embedding. Define $X(\Sigma)$ as the monoid scheme
\[
\underset{\sigma \in \mathcal{U}}{\textup{colim}} \MSpec \F_1[S_\sigma].
\]
Let $A$ be a pointed monoid. The monoid scheme 
\[
X(\Sigma) \times_{\F_1} \MSpec A
\]
has an open covering given by $\{\MSpec A[S_\sigma]\}_{\sigma \in \Sigma}$ (\textit{cf.} \cite[Prop 3.1]{Deitmar05}).

Let $\F$ be a pointed group that is also a domain, $\sigma$ a rational polyhedral cone and $\tau$ a face of $\sigma^\vee$. Note that the set $\mathfrak{p}_\tau := (\F[S_\sigma] \backslash \F[S_\tau^\vee]) \cup \{0\}$ is an $\F$-prime of $\F[S_\sigma]$. The next result shows that every prime in $\MSpec_\F \F[S_\sigma]$ has this form.

\begin{lemma}
\label{bijection primes-faces}
Let $\F$ be a pointed group that is also a domain and $\sigma$ a rational polyhedral cone. Then the map $\tau \mapsto \mathfrak{p}_\tau$ is a bijection between the set of faces of $\sigma^\vee$ and $\MSpec_\F \F[S_\sigma]$ that reverses inclusions.
\end{lemma}

\begin{proof}
It is enough to show that the map $\{\textup{faces of } \sigma^\vee\} \rightarrow \MSpec_\F \F[S_\sigma]$ above is surjective.

Let $\mathfrak{p} \in \MSpec_\F \F[S_\sigma]$. We use $S_\sigma$ to denote the semigroup $\F_1[S_\sigma]\backslash\{0\} \subseteq \F[S_\sigma]$ and define $\mathfrak{q}:= \mathfrak{p}\cap S_\sigma$. As $S_\sigma \backslash \mathfrak{q}$ is a multiplicative subset, by \cite[Prop. 2.2.1]{Ogus18}, there exists a morphism of semigroups $h: S_\sigma \rightarrow \N$ such that $h^{-1}(0) = S_\sigma \backslash \mathfrak{q}$. We have a unique extension of $h$ to a morphism of groups $\tilde{h}: M \rightarrow \Z$ and to a map of $\R$-algebras $f: M_\R \rightarrow \R$. One knows that there exists a unique $u \in N_\R$ such that $f(m) = \langle m, u \rangle$ for all $m \in M_\R$. Note that $u \in \sigma$ and define $\tau := \{m \in \sigma^\vee \mid \langle m,  u \rangle = 0\}$, which is a face of $\sigma^\vee$. Next we show that $\mathfrak{p} = \mathfrak{p}_\tau$.

If $\lambda.\chi^m \in \mathfrak{p} \backslash \{0\}$, one has $m \in \mathfrak{q}$, which implies $h(m) \neq 0$ and, consequently, $\lambda . \chi^m$ not in $\F[S_\tau^\vee]$. Thus $\mathfrak{p} \subseteq \mathfrak{p}_\tau$. If $\alpha.\chi^s \in \mathfrak{p}_\tau \backslash \{0\}$, then $s \notin S_\tau^\vee = S_\sigma \cap \tau$, which implies $h(s) \neq 0$ and, consequently, $s \in  \mathfrak{q}$. Therefore $\mathfrak{p}_\tau \subseteq \mathfrak{p}$.
\end{proof}

\begin{lemma}
\label{faces with same dim are equal}
Let $\sigma$ be a rational polyhedral cone and $\tau \subseteq \alpha$ faces of $\sigma^\vee$ such that $\dim \, \tau = \dim \, \alpha$. Then $\tau = \alpha$.
\end{lemma}

\begin{proof}
Let $t \in (\sigma^\vee)^\vee$ such that $\tau = \sigma^\vee \cap H_t$. As $\dim \, \tau = \dim \, \alpha$, one has $\textup{span} \, \tau = \textup{span} \, \alpha$. As $H_t$ is a linear subspace, one has $\alpha \subseteq H_t$ if, and only if, $\textup{span} \, \alpha \subseteq H_t$. Note that $\tau \subseteq \alpha \subseteq \textup{span} \, \alpha \cap \sigma^\vee = \textup{span} \, \tau \cap \sigma^\vee \subseteq H_t \cap \sigma^\vee = \tau$. Thus $\tau = \alpha$.
\end{proof}

\begin{lemma}
\label{faces for all dimensions}
Let $\sigma$ be a rational polyhedral cone and $\tau \subseteq \alpha$ faces of $\sigma^\vee$. Let $\ell$ be an integer such that $\dim \, \tau \leq \ell \leq \dim \, \alpha$. Then there exists a face $\delta$ of $\sigma^\vee$ of dimension $\ell$ with $\tau \subseteq \delta \subseteq \alpha$.
\end{lemma}

\begin{proof}
If $\beta$ is a face of $\alpha$ and $\delta$ is a face of $\beta$, then $\delta$ is a face of $\alpha$. Thus it is enough to prove the result for $\ell = \dim \alpha - 1$, which is done in \cite[p. 10]{Fulton93}.
\end{proof}

\begin{lemma}
\label{torsion-free quotient by sum}
Let $G$ be a torsion-free abelian group of finite rank and $H_1, H_2$ subgroups such that $H_1 \cap H_2 = \{0\}$ and $G/H_i$ is torsion-free for $i=1, 2$. Then $G/ (H_1\oplus H_2)$ is torsion free.
\end{lemma}

\begin{proof}
Let $\pi_i: G \rightarrow G/H_i$ be the natural projection for $i = 1,2$. As $G / H_1$ is torsion-free, it is also free, thus there exists a morphism $\theta: G/H_1 \rightarrow G$ such that $\pi_1 \circ \theta = \textup{id}_{G/H_1}$. For each $g \in G$ there exists $h_g \in H_1$ satisfying $\theta(g + H_1) = g + h_g$, thus the kernel of $\pi_2 \circ \theta$ is $(H_1 \oplus H_2) / H_1$.

As $G/H_2$ is torsion-free and one has the injection $(G/H_1) / \ker (\pi_2 \circ \theta) \hookrightarrow G/H_2$ and the isomorphism 
\[
(G/H_1) / \big( (H_1 \oplus H_2) / H_1 \big) \simeq G / (H_1\oplus H_2), 
\]
the result follows.
\end{proof}

\begin{thm}
\label{dimension - affine case}
Let $\F$ be an algebraically closed pointed group and $\sigma \subseteq N_\R$ a strongly convex rational polyhedral cone. Then $\SCong_\F \F[S_\sigma]$ is a catenary space of Krull dimension $n$ (the rank of $N$). 
\end{thm}

\begin{proof}
Firstly, we define two functions $\mathfrak{N}$ and $\mathfrak{T}$ from $\SCong_\F \F[S_\sigma]$ to $\N$ whose sum is analogous to the height of a prime ideal in a ring. Let $\mathfrak{c}$ be an $\F$-congruence on $\F[S_\sigma]$. By \autoref{bijection primes-faces}, there exists a unique face $\tau_\mathfrak{c}$ of $\sigma^\vee$ such that $\mathfrak{p}_{\tau_\mathfrak{c}} = I_\mathfrak{c}$. Define  $\mathfrak{N}(\mathfrak{c})$ as $n- \dim \, \tau_\mathfrak{c}$. 

Let $\mathfrak{d}_\mathfrak{c}$ be the strong $\F$-congruence on $\F[S_{\tau_\mathfrak{c}}^\vee] \simeq \F[S_\sigma] / I_\mathfrak{c}$ induced by $\mathfrak{c}$. Let $\Lambda_\mathfrak{c}$ be the linear subspace of $M_\R$ generated by $\tau_\mathfrak{c}$. As $I_{\mathfrak{d}_\mathfrak{c}} = \{0\}$ and $\F[S_{\Lambda_\mathfrak{c}}^\vee] = \textup{Frac}(\F[S_{\tau_\mathfrak{c}}^\vee])$, one has $\mathfrak{d}_\mathfrak{c} = f^*(\mathfrak{a}_\mathfrak{c})$, where $f: \F[S_{\tau_\mathfrak{c}}^\vee] \rightarrow \F[S_{\Lambda_\mathfrak{c}}^\vee]$ is the inclusion and $\mathfrak{a}_\mathfrak{c}:= f_*(\mathfrak{d}_\mathfrak{c})  \in \SCong \F[S_{\Lambda_\mathfrak{c}}^\vee]$. Let $\equiv$ be the equivalence relation on $S_{\Lambda_\mathfrak{c}}^\vee$ given by $a\equiv b$ if $\chi^a \sim_{\mathfrak{a}_\mathfrak{c}} \lambda \chi^b$ for some $\lambda \in \F^\times$ and define the subgroup $H_\mathfrak{c}:= \{a \in S_{\Lambda_\mathfrak{c}}^\vee \mid a \equiv 0\}$ of $S_{\Lambda_\mathfrak{c}}^\vee$. Note that 
\[
\mathfrak{a}_\mathfrak{c} = \langle \chi^a \sim \lambda \mid a \in H_\mathfrak{c}  \text{ and } \chi^a \sim_{\mathfrak{a}_\mathfrak{c}} \lambda \rangle.
\]
Define $\mathfrak{T}(\mathfrak{c})$ as $\textup{rk}(H_\mathfrak{c})$ and note that $\mathfrak{t}(\mathfrak{c}) \leq \textup{rk}(S_{\Lambda_\mathfrak{c}}^\vee) =  \textup{dim} \, \tau_\mathfrak{c}$.

Next we show that $S_{\Lambda_\mathfrak{c}}^\vee / H_\mathfrak{c}$ is torsion-free. Let $y \in S_{\Lambda_\mathfrak{c}}^\vee$ and $r > 0$ such that $ry \in H_\mathfrak{c}$. Then there exists $\lambda \in \F^\times$ such that $(\chi^y)^r \sim_{\mathfrak{a}_\mathfrak{c}} \lambda$. As $\F$ is algebraically closed and $\mathfrak{a}_\mathfrak{c}$ is strong, there exists $\zeta \in \F$ such that $\zeta^r = \lambda$ and $\chi^y\sim_{\mathfrak{a}_\mathfrak{c}} \zeta$, thus $y \in H_\mathfrak{c}$.

Let $\mathfrak{c}_1$ and $\mathfrak{c}_2$ be $\F$-congruences on $\F[S_\sigma]$ such that $\mathfrak{c}_1 \subseteq \mathfrak{c}_2$. Let $\mathcal{C}$ be a maximal chain in $\SCong_\F \F[S_\sigma]$ with minimum $\mathfrak{c}_1$ and maximum $\mathfrak{c}_2$. We claim that $\mathcal{C}$ is finite. Note that $\mathcal{C}_I := \{I_\mathfrak{c} \mid \mathfrak{c} \in \mathcal{C}\}$  is a chain in $\MSpec_\F \F[S_\sigma]$, thus $\# \mathcal{C}_I \leq n$, and given $J \in \mathcal{C}_I$, one has that $\mathcal{H}_J := \{ H_\mathfrak{c} \mid \mathfrak{c} \in \mathcal{C} \text{ with } I_\mathfrak{c} = J\}$ is a chain whose length is bounded by $\textup{dim} \, \tau$, where $\tau$ is the unique face of $\sigma^\vee$ satisfying $\mathfrak{p}_\tau = J$. Next we show that the length of $\mathcal{C}$ is $\big(\mathfrak{N}(\mathfrak{c}_2) + \mathfrak{T}(\mathfrak{c}_2)\big) - \big(\mathfrak{N}(\mathfrak{c}_1) + \mathfrak{T}(\mathfrak{c}_1)\big)$.

It is enough to show that if $\mathfrak{r} \subsetneq \mathfrak{r}'$ are two consecutive congruences in $\mathcal{C}$, then 
\[
\big( \mathfrak{T}(\mathfrak{r'}) - \mathfrak{T}(\mathfrak{r})\big) + \big( \mathfrak{N}(\mathfrak{r'}) - \mathfrak{N}(\mathfrak{r}) \big) = 1.
\]

Let $\pi: \F[S_\sigma] \rightarrow \F[S_{\tau_\mathfrak{r}}^\vee] \simeq \F[S_\sigma] / I_\mathfrak{r}$ be the projection. If $\mathfrak{c} \in \SCong_\F \F[S_\sigma]$ satisfies $I_\mathfrak{r} \subseteq I_\mathfrak{c}$, then $\pi_*(\mathfrak{c}) \in \SCong_\F \F[S_{\tau_\mathfrak{r}}^\vee]$, $\mathfrak{N}(\pi_*(\mathfrak{c})) = \mathfrak{N}(\mathfrak{c}) - \mathfrak{N}(\mathfrak{r})$ and $\mathfrak{T}(\pi_*(\mathfrak{c})) = \mathfrak{T}(\mathfrak{c})$, thus we can assume $I_\mathfrak{r} = \{0\}$. Therefore $\Lambda_\mathfrak{r} = M_\R$ and we have two cases to analyze:

\medskip\noindent
\textbf{Case 1:} $I_\mathfrak{r'} = \{0\}$. Then $H_\mathfrak{r} \subseteq H_\mathfrak{r'}$, which implies 
$\mathfrak{T}(\mathfrak{r}) \leq \mathfrak{T}(\mathfrak{r'})$. One has the injection $H_\mathfrak{r'} / H_\mathfrak{r} \hookrightarrow S_{\Lambda_\mathfrak{r}}^\vee / H_\mathfrak{r}$, which shows that $H_\mathfrak{r'} / H_\mathfrak{r}$ is torsion-free. If $\mathfrak{T}(\mathfrak{r}) = \mathfrak{T}(\mathfrak{r'})$, then $H_\mathfrak{r} = H_\mathfrak{r'}$ and we conclude that $\mathfrak{r} = \mathfrak{r'}$. Thus $\mathfrak{T}(\mathfrak{r}) < \mathfrak{T}(\mathfrak{r'})$. 

If $\mathfrak{T}(\mathfrak{r}) \leq \mathfrak{T}(\mathfrak{r'}) -2$, let $s := \mathfrak{T}(\mathfrak{r'})$, $\ell := \mathfrak{T}(\mathfrak{r})$ and $\{b_1, \dotsc, b_s\}$ a $\Z$-basis of $H_\mathfrak{r'}$ such that $\{b_1, \dotsc, b_\ell\}$ is a $\Z$-basis of $H_\mathfrak{r}$. Then, for each $1 \leq i \leq s$, there exists a unique $\lambda_i \in \F^\times$ such that $b_i \sim_{\mathfrak{a}_\mathfrak{r'}} \lambda_i$. As $\mathfrak{r} \subseteq \mathfrak{r'}$, one has $b_j \sim_{\mathfrak{a}_\mathfrak{r}} \lambda_j$ for $1 \leq j \leq \ell$. Let $\mathfrak{a} \in \SCong_\F \F[S_{\Lambda_\mathfrak{r}}^\vee]$ be the congruence $\langle b_i \sim \lambda_i \mid 1 \leq i \leq \ell+1 \rangle$ and define $\mathfrak{c} := g^*(\mathfrak{a})$, where $g: \F[S_\sigma] \rightarrow \F[S_{\Lambda_\mathfrak{r}}^\vee]$ is the inclusion. Thus $\mathfrak{r} \subsetneq \mathfrak{c} \subsetneq \mathfrak{r'}$, which is a contradiction. Therefore, $\mathfrak{T}(\mathfrak{r}) = \mathfrak{T}(\mathfrak{r'}) - 1$.

\medskip\noindent
\textbf{Case 2:} $I_\mathfrak{r'} \neq \{0\}$. Suppose that $H_\mathfrak{r'} \not\subset H_\mathfrak{r}$, which is equivalent to $H_\mathfrak{r} \cap H_\mathfrak{r'} \subsetneq H_\mathfrak{r'}$. One has that $H_\mathfrak{r'} / H_\mathfrak{r} \cap H_\mathfrak{r'}$ is torsion free, because $S_{\Lambda_\mathfrak{r'}}^\vee \subsetneq S_{\Lambda_\mathfrak{r}}^\vee$ and $H_\mathfrak{r'} / H_\mathfrak{r} \cap H_\mathfrak{r'} \hookrightarrow S_{\Lambda_\mathfrak{r}}^\vee / H_\mathfrak{r}$. Let $\ell := \textup{rk}(H_\mathfrak{r} \cap H_\mathfrak{r'})$ and $s:= \mathfrak{T}(\mathfrak{r'})$. Then there exists a $\Z$-basis $\{b_1, \dotsc, b_s\}$ of $H_\mathfrak{r'}$ such that $\{b_1, \dotsc, b_\ell\}$ is a $\Z$-basis of $H_\mathfrak{r} \cap H_\mathfrak{r'}$. For each $1\leq i \leq s$, let $\lambda_i \in \F^\times$ such that $b_i \sim_{\mathfrak{a}_\mathfrak{r'}} \lambda_i$. Let $\mathfrak{b} \in \SCong_\F \F[S_{\Lambda_\mathfrak{r'}}^\vee]$ be the congruence $\langle b_i \sim \lambda_i \mid 1 \leq i \leq \ell \rangle$ and define $\mathfrak{c} := p^*(\mathfrak{b})$, where $p$ is the composition
\[
\F[S_\sigma] \relbar\joinrel\twoheadrightarrow \F[S_\sigma] / I_\mathfrak{r'} \simeq \F[S_{\tau_\mathfrak{r'}}^\vee] \longhookrightarrow \F[S_{\Lambda_\mathfrak{r'}}^\vee].
\]
Note that $I_\mathfrak{c} = I_\mathfrak{r'}$ and $H_\mathfrak{c} = H_\mathfrak{r} \cap H_\mathfrak{r'}$, thus $\mathfrak{r} \subsetneq \mathfrak{c} \subsetneq \mathfrak{r'}$, which is a contradiction. Therefore $H_\mathfrak{r'} \subseteq H_\mathfrak{r}$.

Let $s:=  \mathfrak{T}(\mathfrak{r'})$, $\{x_1, \dotsc, x_n\}$ a $\Z$-basis of $S_{\Lambda_\mathfrak{r}}^\vee$ such that $\{x_1, \dotsc, x_s\}$ is a $\Z$-basis of $G:= H_\mathfrak{r'}$ and $\lambda_i \in \F^\times$ such that $x_i \sim_{\mathfrak{a}_\mathfrak{r'}} \lambda_i$ for $1\leq i \leq s$. The map of groups
\[
\begin{array}{ccl}
S_{\Lambda_\mathfrak{r}}^\vee & \longrightarrow &(\F[S_{\Lambda_\mathfrak{r}}^\vee / G])^\times\\
 & & \\
x_i &\longmapsto & \smash{\left\{\begin{array}{ll}
      \lambda_i &\text{if } 1\leq i \leq s\\
    {}\chi^{\overline{x_i}} &\text{if } s+1 \leq i \leq n
    \end{array}\right.}
\end{array}
\]
\\
\noindent
induces a surjective morphism of $\F$-algebras $\varphi: \F[S_{\Lambda_\mathfrak{r}}^\vee] \twoheadrightarrow \F[S_{\Lambda_\mathfrak{r}}^\vee / G]$ with congruence kernel $\iota_*(\mathfrak{a}_\mathfrak{r'}) \subseteq \mathfrak{a}_\mathfrak{r}$, where $\iota:  \F[S_{\Lambda_\mathfrak{r'}}^\vee] \rightarrow \F[S_{\Lambda_\mathfrak{r}}^\vee]$ is the inclusion. As $S_{\Lambda_\mathfrak{r}}^\vee / S_{\Lambda_\mathfrak{r'}}^\vee$ and $S_{\Lambda_\mathfrak{r'}}^\vee / G$ are free, one has that $S_{\Lambda_\mathfrak{r}}^\vee / G$ is free, which implies that $\overline{S_\sigma} := \big\{\overline{x} \in S_{\Lambda_\mathfrak{r}}^\vee / G\; \big| \; x \in S_\sigma\big\}$ is an affine semigroup. Note that $\varphi(\F[S_\sigma]) = \F[\overline{S_\sigma}]$ and define
\[
\begin{array}{cccc}
\psi: & \F[S_\sigma] & \longrightarrow & \F[\overline{S_\sigma}]\\
 & a & \longmapsto & \varphi(a),
\end{array}
\]
$\mathfrak{s} := \psi_*(\mathfrak{r})$ and $\mathfrak{s'} := \psi_*(\mathfrak{r'})$. Let $\widetilde{M}$ be the group $S_{\Lambda_\mathfrak{r}}^\vee / G$ and $\widetilde{N}:= \textup{Hom}(\widetilde{M}, \Z)$ its dual. Let $y_1, \dotsc, y_w \in M$ be generators of $S_\sigma$ as a monoid and $\rho:= \textup{Cone}(\overline{y_1}, \dotsc, \overline{y_w})^\vee \subseteq \widetilde{N}_\R$.  Note that $\overline{S_\sigma} = S_\rho$ and $\textup{congker}(\psi) \subseteq \mathfrak{r}$, which implies $\mathfrak{s}, \mathfrak{s'} \in \SCong_\F \F[S_\rho]$. As $H_\mathfrak{s} = H_\mathfrak{r} / G$ and $H_\mathfrak{s'} = H_\mathfrak{r'} / G = \{0\}$, it is enough to show that
\[
\mathfrak{N}(\mathfrak{s'}) - \mathfrak{T}(\mathfrak{s}) = 1.
\]

Define $m:= \textup{rk}(\widetilde{M})$, $d:= \mathfrak{T}(\mathfrak{s})$ and $e:= m - \mathfrak{N}(\mathfrak{s'})$. Let $\{u_1, \dotsc, u_m\}$ be a $\Z$-basis of $\widetilde{M}$ such that $\{u_1, \dotsc, u_d\}$ is a $\Z$-basis of $H_\mathfrak{s}$, and $\beta_i \in \F^\times$ with $\chi^{u_i} \sim_{\mathfrak{a}_\mathfrak{s}} \beta_i$ for $1\leq i \leq d$. Analogous to what is done above, there exists a surjective morphism of $\F$-algebras 
\[
\begin{array}{cccl}
\widetilde{\varphi}: & \F[\widetilde{M}] & \longrightarrow &\F[\widetilde{M} / H_\mathfrak{s}]\\
 & & \\
&\chi^{u_i} &\longmapsto & \smash{\left\{\begin{array}{ll}
      \beta_i &\text{if } 1\leq i \leq d\\
    {}\chi^{\overline{u_i}} &\text{if } d+1 \leq i \leq m
    \end{array}\right.}
\end{array}
\]
\\noindent such that $\congker (\widetilde{\varphi}) = \mathfrak{a}_\mathfrak{s}$ and $\widetilde{\varphi}(\F[S_\rho]) = \F[\overline{S_\rho}]$, where $\overline{S_\rho}$ is the affine semigroup $\{ \overline{y} \in \widetilde{M} / H_\mathfrak{s} \mid y \in S_\rho\}$. As the composition 
\[
\F[S_{\tau_\mathfrak{s'}}^\vee] \longhookrightarrow \F[S_\rho] \relbar\joinrel\twoheadrightarrow \F[S_\rho] / \mathfrak{s} \relbar\joinrel\twoheadrightarrow \F[S_\rho] / \mathfrak{s'}
\]
is an isomorphism, one has the commutative diagram
\[
\begin{tikzcd}
\F[S_{\tau_\mathfrak{s'}}^\vee]  \arrow[r, hook] \arrow[d, hook] & \F[S_\rho] / \mathfrak{s} \arrow[d, hook]\\
\F[S_{\Lambda_\mathfrak{s'}}^\vee]  \arrow[r, hook]& \F[\widetilde{M}] / \mathfrak{a}_\mathfrak{s} \simeq \F[\widetilde{M} / H_\mathfrak{s}],
\end{tikzcd}
\]
which shows that $S_{\Lambda_\mathfrak{s'}}^\vee \cap H_\mathfrak{s} = \{0\}$. By \autoref{torsion-free quotient by sum}, there exists a $\Z$-basis $\{c_1, \dotsc, c_m\}$ of $\widetilde{M}$ such that $c_i = u_i$ for $1 \leq i \leq d$ and $\{c_{d+1}, \dotsc, c_{d+e}\}$ is a $\Z$-basis of $S_{\Lambda_\mathfrak{s'}}^\vee$. Therefore $d+e \leq m$, \textit{i.e.}, $\mathfrak{T}(\mathfrak{s}) \leq \mathfrak{N}(\mathfrak{s'})$.  Next we show that $\mathfrak{T}(\mathfrak{s}) = \mathfrak{N}(\mathfrak{s'}) - 1$.

Suppose that $\mathfrak{T}(\mathfrak{s}) = \mathfrak{N}(\mathfrak{s'})$. Then the map $\F[S_{\Lambda_\mathfrak{s'}}^\vee] \rightarrow \F[\widetilde{M} / H_\mathfrak{s}]$ is surjective. Let $Q$ be the submonoid of $\widetilde{M}$ generated by $S_\rho$ and $S_{\Lambda_\mathfrak{s'}}^\vee$. As $\F[Q]$ is a localization of $\F[S_\rho]$ by elements not in $I_\mathfrak{s'}$ or $I_\mathfrak{s} = \{0\}$, one has that $q_*(\mathfrak{s})$ and $q_*(\mathfrak{s'})$ are $\F$-congruences on $\F[Q]$, where $q: \F[S_\rho] \rightarrow \F[Q]$ is the inclusion map. Note that $\textup{congker}(\widetilde{\varphi}|_{\F[Q]}) = q_*(\mathfrak{s})$. As the diagram
\[
\begin{tikzcd}
\F[Q] \arrow[rdd, "\widetilde{\varphi}"] \arrow[dd, two heads] &\\
&&\\
\F[Q] / q_*(\mathfrak{s}) \arrow[dd, two heads] \arrow[r, hook]&  \F[\widetilde{M} / H_\mathfrak{s}] \\
&&\\
\F[Q] / q_*(\mathfrak{s'}) \simeq \F[S_{\Lambda_\mathfrak{s'}}^\vee] \arrow[uur, two heads] \arrow[r, hook] & \F[Q] \arrow[uu, "\widetilde{\varphi}", swap]
\end{tikzcd}
\]
commutes, one has $q_*(\mathfrak{s'}) = \textup{congker}(\widetilde{\varphi}|_{\F[Q]}) = q_*(\mathfrak{s})$, which is a contradiction.

Suppose that $\mathfrak{T}(\mathfrak{s}) \leq \mathfrak{N}(\mathfrak{s'}) - 2$. Let $\widetilde{\psi} := \widetilde{\varphi}|_{\F[S_\rho]}$ and $\mathfrak{u} := \langle \overline{c_m} \sim 1 \rangle$ on $\F[\widetilde{M} / H_\mathfrak{s}]$. As $\mathfrak{u}$ is an $\F$-congruence, one has that $\mathfrak{v} := \widetilde{\psi}^*(\mathfrak{u})$ is an $\F$-congruence on $\F[S_\rho]$ satisfying $\mathfrak{s} \subsetneq \mathfrak{v} \subsetneq \mathfrak{s'}$, which is a contradiction.

The topological dimension of $\SCong_\F \F[S_\sigma]$ is $n$ because for any maximal congruence $\mathfrak{c}$ in $\SCong_\F \F[S_\sigma]$, one has $\F[S_\sigma] / \mathfrak{c} \simeq \F$ and 
\[
\mathfrak{T}(\mathfrak{c}) = \textup{rk}(H_\mathfrak{c}) = \textup{rk}(S_{\Lambda_\mathfrak{c}}^\vee) = \dim \tau_\mathfrak{c} = n - \mathfrak{N}(\mathfrak{c}).
\]
\end{proof}

\begin{df}
Let $k$ be a pointed monoid and $X$ a $k$-scheme. We define the \textit{$k$-dimension} of $X$ as the Krull dimension of the topological space $\SCong_k X$, and denote it by $\textup{dim}_k X$.
\end{df}

\begin{thm}
\label{thm-toric dimension}
Let $\F$ be an algebraically closed pointed group, $\Sigma$ a fan and 
\[
X := X(\Sigma) \times_{\F_1}\MSpec \F
\]
the toric $\F$-scheme associated to $\Sigma$. Then $\SCong_\F X$ is a catenary space of dimension $\dim_\F X = \dim \Sigma$.
\end{thm}

\begin{proof}
By \autoref{strong covering}, $\{ \SCong_\F \F[S_\sigma] \mid \sigma \in \Sigma\}$ is an open covering of $\SCong_\F X$. By \autoref{dimension - affine case}, each $\SCong_\F \F[S_\sigma]$ is catenary of dimension $n$, which implies $\dim_\F X = n$. By \cite[Lemma 5.11.5]{stacks-project}, $X$ is catenary. 
\end{proof}

\section{Examples: The torus and affine space}

Throughout this section, we fix an algebraically closed pointed group $\F$ and give an explicit description of the strong congruence space of the torus and affine space over $\F$.

\begin{notation}
\textup{For $y = (y_1, \dotsc, y_n) \in \Z^n$, we use $X^y$ to denote the monomial $X_1^{y_1} \cdots X_n^{y_n}$.}
\end{notation}

\subsection{The torus} The \textit{$n$-torus over $\F$} is the affine monoid scheme 
\[
\T_\F^n := \MSpec \F[X_1^\pm, \dotsc, X_n^\pm].
\]
Let $\{z_1, \dotsc, z_c\}$ be part of a $\Z$-basis of $\Z^n$ and $\lambda_1, \dotsc, \lambda_c \in \F^\times$. Define the congruence
\[
\mathfrak{c} := \langle X^{z_i} \sim \lambda_i \mid i = 1, \dotsc, c \rangle
\]
on $\F[X_1^\pm, \dotsc, X_n^\pm]$. Let $z_{c+1}, \dotsc, z_n \in \Z^n$ such that $\{z_1, \dotsc, z_n\}$ is a $\Z$-basis. The map
 \[
  \begin{array}{cccc}
  \varphi: & \F[X_1^\pm, \dotsc, X_n^\pm] & \longrightarrow & \F[X_1^\pm, \dotsc, X_n^\pm] \\
           & X_i & \longmapsto & X^{z_i}
  \end{array}
 \]
defines an isomorphism of $\F$-algebras. As
\[
\varphi^*(\mathfrak{c}) = \langle X_i \sim \lambda_i \mid i = 1, \dotsc, c \rangle,
\]
one has $\F[X_1^\pm, \dotsc, X_n^\pm] / \mathfrak{c} \simeq \F[X_1^\pm, \dotsc, X_n^\pm] / \varphi^*(\mathfrak{c}) \simeq \F[X_{c+1}^\pm, \dotsc, X_n^\pm]$. Therefore $\mathfrak{c}$ is an $\F$-congruence. 

The next result shows that every $\F$-congruence on $\F[X_1^\pm, \dotsc, X_n^\pm]$ has this form.

\begin{prop}
\label{characterization torus}
Let $\mathfrak{c} \in \SCong_\F \T_\F^n$. Then there exists an integer $0\leq c \leq n$, constants $\lambda_1, \dotsc, \lambda_c \in \F^\times$ and $z_1, \dotsc, z_c \in \Z^n$ part of a $\Z$-basis such that 
\[
\mathfrak{c} = \langle X^{z_i} \sim \lambda_i \mid i = 1, \dotsc, c \rangle.
\]
\end{prop}

\begin{proof}
As $\F[X_1^\pm, \dotsc, X_n^\pm]$ is a pointed group, one has $I_\mathfrak{c} = \{0\}$. Let 
\[H := \{z \in \Z^n \mid X^z \sim_{\mathfrak{c}} \lambda \text{ for some } \lambda \in \F^\times\}. 
\]
Note that $H$ is a subgroup of $\Z^n$. We show that $\Z^n / H$ is torsion-free. Let $z \in \Z^n$ and $m > 0$ such that $mz \in H$. Then there exists $\lambda \in \F^\times$ such that $(X^z)^m \sim_\mathfrak{c} \lambda$. Let $\alpha \in \F$ such that $\alpha^m = \lambda$ and $\zeta \in \F$ be a primitive $m^{\textup{th}}$-root of $1$, \textit{i.e.}, a generator of the group $\mu_m := \{w \in \F \mid w^m = 1\}$. As $\mathfrak{c}$ is an $\F$-congruence, the sets
\[
\big\{x \in \F[X_1^\pm, \dotsc, X_n^\pm] / \mathfrak{c} \; \big| \; x^m = \lambda\big\} \quad \text{and} \quad \big\{[\zeta^i \alpha] \; \big| \; i = 0, \dotsc, m-1\big\}
\]
coincide. As $[X^z]^m = \lambda$, there exists $0 \leq j \leq m-1$ such that $X^z \sim_\mathfrak{c} \zeta^j \alpha$. Thus $z \in H$. Then there exists a $\Z$-basis $\{z_1, \dotsc, z_c\}$ of $H$ that can be completed to a $\Z$-basis of $\Z^n$. 

As $\mathfrak{c}$ is an $\F$-congruence and $\F$ is a pointed group, there exist unique $\lambda_1, \dotsc, \lambda_c \in \F^\times$ such that $X^{z_i} \sim_\mathfrak{c} \lambda_i$ for $i = 1, \dotsc, c$. Next we show that $\mathfrak{c}$ is equal to
\[
\mathfrak{d} := \langle X^{z_i} \sim \lambda_i \mid i = 1, \dotsc, c \rangle.
\]
Let $g, h \in (\F[X_1^\pm, \dotsc, X_n^\pm])^\times$ such that $g \sim_\mathfrak{c} h$. Then $gh^{-1} \sim_\mathfrak{c} 1$. Let $\omega \in \F^\times$ and $z \in \Z^n$ such that $gh^{-1} = \omega X^z$. As $X^z \sim_\mathfrak{c} \omega^{-1}$, one has $X^z \in H$. Thus there exist $m_1, \dotsc, m_c \in \Z$ such that $z = \sum m_i z_i$, which implies $\omega^{-1} \sim_\mathfrak{c} X^z \sim_\mathfrak{c} \prod \lambda_i^{m_i}$. As $\mathfrak{c}$ is an $\F$-congruence, one has $\omega^{-1} = \prod \lambda_i^{m_i}$. Note that $X^z \sim_\mathfrak{d} \prod \lambda_i^{m_i}$, which implies
\[
g \sim_\mathfrak{d} gh^{-1} h \sim_\mathfrak{d} \omega X^z h \sim_\mathfrak{d} \omega \omega^{-1} h \sim_\mathfrak{d} h.
\]
Therefore $g \sim_\mathfrak{d} h$.
\end{proof}

\subsection{Affine space}
The \textit{affine $n$-space over $\F$} is the affine monoid scheme 
\[
\A_\F^n := \MSpec \F[X_1, \dotsc, X_n].
\]
We show the analogue of \autoref{characterization torus} in the case of the affine space. For a subset $L \subseteq \{1, \dotsc, n\}$, let $\mathfrak{p}_L$ be the prime ideal $\langle X_i \mid i \in L \rangle \in \A_\F^n$. Note that 
\[
\F[X_1, \dotsc, X_n] / \mathfrak{p}_L \simeq \F[X_i \mid i \notin L]. 
\]
Thus, for an $\F$-congruence $\mathfrak{d}$ on $\F[X_i^\pm \mid i \notin L] = \textup{Frac}(\F[X_i \mid i \notin L])$, one has that $\varphi^*(\mathfrak{d})$ is an $\F$-congruence on $\F[X_1, \dotsc, X_n]$, where $\varphi$ is the composition
\[
\F[X_1, \dotsc, X_n]  \relbar\joinrel\twoheadrightarrow \F[X_1, \dotsc, X_n] / \mathfrak{p}_L \simeq \F[X_i \mid i \notin L] \longhookrightarrow \F[X_i^\pm \mid i \notin L].
\]

By \autoref{characterization torus}, $\mathfrak{d} = \langle X^{z_i} \sim \lambda_i \mid i = 1, \dotsc, c \rangle$, where $\{z_1, \dotsc, z_c\} \subset \Z^{n-\#L}$ is part of a $\Z$-basis, which implies that 
\[
\varphi^*(\mathfrak{d}) = \langle \{X_i \sim 0 \mid i \in L \} \cup \{a \sim b \mid a, b \in \F[X_i \mid i \notin L] \text{ with } a \sim_\mathfrak{d} b\} \rangle.
\]
The next result shows that every $\mathfrak{c} \in \SCong_\F \A_\F^n$ has this form.

\begin{prop}
\label{characterization affine}
Let $\mathfrak{c} \in \SCong_\F \A_\F^n$. Then there exists a subset $L \subseteq \{1, \dotsc, n\}$, an integer $0\leq c \leq n-\#L$, constants $\lambda_1, \dotsc, \lambda_c \in \F^\times$ and $\{z_1, \dotsc, z_c\} \subseteq \Z^{n-\#L}$ part of a $\Z$-basis such that
\[
\mathfrak{c} = \langle \{X_i \sim 0 \mid i \in L \} \cup \{a \sim b \mid a, b \in \F[X_i \mid i \notin L] \text{ with } a \sim_\mathfrak{d} b\} \rangle,
\]
where $\mathfrak{d}$ is the congruence $\langle X^{z_i} \sim \lambda_i \mid i = 1, \dotsc, c \rangle$ on $\F[X_i^\pm \mid i \notin L]$.
\end{prop}

\begin{proof}
As $I_\mathfrak{c} \in \A_\F^n$, there exists $L \subseteq \{1, \dotsc, n\}$ such that $I_\mathfrak{c} = \mathfrak{p}_L$. Let 
\[\pi: \F[X_1, \dotsc, X_n]  \relbar\joinrel\twoheadrightarrow \F[X_1, \dotsc, X_n] / \mathfrak{p}_L \simeq \F[X_i \mid i \notin L]
\]
be the projection, 
\[
\iota: \F[X_i \mid i \notin L] \longhookrightarrow \F[X_i^\pm \mid i \notin L]
\]
the localization map, $\mathfrak{r}:= \pi_*(\mathfrak{c})$ and $\mathfrak{d} := \iota_*(\pi_*(\mathfrak{c}))$. Note that $\mathfrak{r} \in \SCong_\F \A_\F^{n-\#L}$ and $\mathfrak{c} = \pi^*(\mathfrak{r})$. As $\mathfrak{r}$ satisfies $I_\mathfrak{r} = \{0\}$, by \autoref{k-localization}, one has $\mathfrak{d} \in \SCong_\F \T_\F^{n-\#L}$ and $\mathfrak{r} = \iota^*(\mathfrak{d})$. Thus  the result follows from the discussion above, because $\mathfrak{c} = (\iota \circ \pi)^*(\mathfrak{d})$.
\end{proof}


\bibliographystyle{alpha}
\bibliography{dimension}

\end{document}